\documentclass[11pt]{amsart}
\usepackage{amssymb}
\usepackage{amsmath}
\theoremstyle{plain}
\newtheorem{thm}{Theorem}[section]

\newtheorem{theorem}[thm]{Theorem}
\newtheorem{lemma}[thm]{Lemma}
\newtheorem{corollary}[thm]{Corollary}
\newtheorem{proposition}[thm]{Proposition}

\newtheorem*{myprop1}{Proposition}

\newtheorem*{mycor1}{Corollary}

\theoremstyle{definition}

\newtheorem{definition}[thm]{Definition}

\newtheorem{example}[thm]{Example}

\newtheorem{defn-thm}[thm]{Definition-Theorem}

\newtheorem*{myrem}{Remark}

\numberwithin{equation}{thm}

\catcode`\@=11
\def\opn#1#2{\def#1{\mathop{\kern0pt\fam0#2}\nolimits}}
\def\underrightarrow{\mathpalette\underrightarrow@}
\def\underrightarrow@#1#2{\vtop{\ialign{$##$\cr
 \hfil#1#2\hfil\cr\noalign{\nointerlineskip}%
 #1{-}\mkern-6mu\cleaders\hbox{$#1\mkern-2mu{-}\mkern-2mu$}\hfill
 \mkern-6mu{\to}\cr}}}

\def\underleftarrow{\mathpalette\underleftarrow@}
\def\underleftarrow@#1#2{\vtop{\ialign{$##$\cr
 \hfil#1#2\hfil\cr\noalign{\nointerlineskip}#1{\leftarrow}\mkern-6mu
 \cleaders\hbox{$#1\mkern-2mu{-}\mkern-2mu$}\hfill
 \mkern-6mu{-}\cr}}}
\let\amp@rs@nd@\relax
\newdimen\ex@
\newdimen\bigaw@
\newdimen\minaw@
\newdimen\minCDaw@
\newif\ifCD@
\def\minCDarrowwidth#1{\minCDaw@#1}

\def\@CD{\def\A##1A##2A{\llap{$\vcenter{\hbox
 {$\scriptstyle##1$}}$}\Big\uparrow\rlap{$\vcenter{\hbox{%
$\scriptstyle##2$}}$}&&}%
\def\V##1V##2V{\llap{$\vcenter{\hbox
 {$\scriptstyle##1$}}$}\Big\downarrow\rlap{$\vcenter{\hbox{%
$\scriptstyle##2$}}$}&&}%
\def\={&\hskip.5em\mathrel
 {\vbox{\hrule width\minCDaw@\vskip3\ex@\hrule width
 \minCDaw@}}\hskip.5em&}%
\def\verteq{\Big\Vert&&}%
\def\noarr{&&}%
\def\vspace##1{\noalign{\vskip##1\relax}}\relax\iffalse{%
\fi\let\amp@rs@nd@&\iffalse}\fi
 \CD@true\vcenter\bgroup\relax\iffalse{%
\fi\let\\=\cr\iffalse}\fi\tabskip\z@skip\baselineskip20\ex@
 &\hfill$\m@th##$\hfill\cr}
\def\@endCD{\cr\egroup\egroup}
\def\>#1>#2>{\amp@rs@nd@\setbox\z@\hbox{$\scriptstyle
 \;{#1}\;\;$}\setbox\@ne\hbox{$\scriptstyle\;{#2}\;\;$}\setbox\tw@
 \hbox{$#2$}\ifCD@
 \global\bigaw@\minCDaw@\else\global\bigaw@\minaw@\fi
 \ifdim\wd\z@>\bigaw@\global\bigaw@\wd\z@\fi
 \ifdim\wd\@ne>\bigaw@\global\bigaw@\wd\@ne\fi
 \ifCD@\hskip.5em\fi
 \ifdim\wd\tw@>\z@
 \mathrel{\mathop{\hbox to\bigaw@{\rightarrowfill}}\limits^{#1}_{#2}}\else
 \mathrel{\mathop{\hbox to\bigaw@{\rightarrowfill}}\limits^{#1}}\fi
 \ifCD@\hskip.5em\fi\amp@rs@nd@}
\def\<#1<#2<{\amp@rs@nd@\setbox\z@\hbox{$\scriptstyle
 \;\;{#1}\;$}\setbox\@ne\hbox{$\scriptstyle\;\;{#2}\;$}\setbox\tw@
 \hbox{$#2$}\ifCD@
 \global\bigaw@\minCDaw@\else\global\bigaw@\minaw@\fi
 \ifdim\wd\z@>\bigaw@\global\bigaw@\wd\z@\fi
 \ifdim\wd\@ne>\bigaw@\global\bigaw@\wd\@ne\fi
 \ifCD@\hskip.5em\fi
 \ifdim\wd\tw@>\z@
 \mathrel{\mathop{\hbox to\bigaw@{\leftarrowfill}}\limits^{#1}_{#2}}\else
 \mathrel{\mathop{\hbox to\bigaw@{\leftarrowfill}}\limits^{#1}}\fi
 \ifCD@\hskip.5em\fi\amp@rs@nd@}
\newenvironment{CDS}{\@CDS}{\@endCDS}
\def\@CDS{\def\A##1A##2A{\llap{$\vcenter{\hbox
 {$\scriptstyle##1$}}$}\Big\uparrow\rlap{$\vcenter{\hbox{%
$\scriptstyle##2$}}$}&}%
\def\V##1V##2V{\llap{$\vcenter{\hbox
 {$\scriptstyle##1$}}$}\Big\downarrow\rlap{$\vcenter{\hbox{%
$\scriptstyle##2$}}$}&}%
\def\={&\hskip.5em\mathrel
 {\vbox{\hrule width\minCDaw@\vskip3\ex@\hrule width
 \minCDaw@}}\hskip.5em&}
\def\verteq{\Big\Vert&}
\def\novarr{&}
\def\noharr{&&}
\def\SE##1E##2E{\slantedarrow(0,18)(4,-3){##1}{##2}&}
\def\SW##1W##2W{\slantedarrow(24,18)(-4,-3){##1}{##2}&}
\def\NE##1E##2E{\slantedarrow(0,0)(4,3){##1}{##2}&}
\def\NW##1W##2W{\slantedarrow(24,0)(-4,3){##1}{##2}&}
\def\slantedarrow(##1)(##2)##3##4{%
\thinlines\unitlength1pt\lower 6.5pt\hbox{\begin{picture}(24,18)%
\put(##1){\vector(##2){24}}%
\put(0,8){$\scriptstyle##3$}%
\put(20,8){$\scriptstyle##4$}%
\end{picture}}}
\def\vspace##1{\noalign{\vskip##1\relax}}\relax\iffalse{%
\fi\let\amp@rs@nd@&\iffalse}\fi
 \CD@true\vcenter\bgroup\relax\iffalse{%
\fi\let\\=\cr\iffalse}\fi\tabskip\z@skip\baselineskip20\ex@
 &\hfill$\m@th##$\hfill\cr}
\def\@endCDS{\cr\egroup\egroup}
\newdimen\TriCDarrw@
\newif\ifTriV@

\def\@TriCDV{\TriV@true\def\TriCDpos@{6}\@TriCD}
\def\@TriCDA{\TriV@false\def\TriCDpos@{10}\@TriCD}
\def\@TriCD#1#2#3#4#5#6{%
\setbox0\hbox{$\ifTriV@#6\else#1\fi$} \TriCDarrw@=\wd0
\advance\TriCDarrw@ 24pt \advance\TriCDarrw@ -1em
\def\SE##1E##2E{\slantedarrow(0,18)(2,-3){##1}{##2}&}
\def\SW##1W##2W{\slantedarrow(12,18)(-2,-3){##1}{##2}&}
\def\NE##1E##2E{\slantedarrow(0,0)(2,3){##1}{##2}&}
\def\NW##1W##2W{\slantedarrow(12,0)(-2,3){##1}{##2}&}
\def\slantedarrow(##1)(##2)##3##4{\thinlines\unitlength1pt
\lower 6.5pt\hbox{\begin{picture}(12,18)%
\put(##1){\vector(##2){12}}%
\put(-4,\TriCDpos@){$\scriptstyle##3$}%
\put(12,\TriCDpos@){$\scriptstyle##4$}%
\end{picture}}}
\def\={\mathrel {\vbox{\hrule
   width\TriCDarrw@\vskip3\ex@\hrule width
   \TriCDarrw@}}}
\def\>##1>>{\setbox\z@\hbox{$\scriptstyle
 \;{##1}\;\;$}\global\bigaw@\TriCDarrw@
 \ifdim\wd\z@>\bigaw@\global\bigaw@\wd\z@\fi
 \hskip.5em
 \mathrel{\mathop{\hbox to \TriCDarrw@
{\rightarrowfill}}\limits^{##1}}
 \hskip.5em}
\def\<##1<<{\setbox\z@\hbox{$\scriptstyle
 \;{##1}\;\;$}\global\bigaw@\TriCDarrw@
 \ifdim\wd\z@>\bigaw@\global\bigaw@\wd\z@\fi
 \mathrel{\mathop{\hbox to\bigaw@{\leftarrowfill}}\limits^{##1}}
 }
 \CD@true\vcenter\bgroup\relax\iffalse{\fi\let\\=\cr\iffalse}\fi
 \tabskip\z@skip\baselineskip20\ex@
 \ifTriV@
 \halign\bgroup
 &\hfill$\m@th##$\hfill\cr
#1&\multispan3\hfill$#2$\hfill&#3\\
&#4&#5\\
&&#6\cr\egroup%
\else
 \halign\bgroup
 &\hfill$\m@th##$\hfill\cr
&&#1\\%
&#2&#3\\
#4&\multispan3\hfill$#5$\hfill&#6\cr\egroup \fi}
\def\@endTriCD{\egroup}

\newcounter{Myenumi}
\newenvironment{myenumi}%
{\begin{list}{}{\usecounter{Myenumi}%
\settowidth{\leftmargin}{2.n}\settowidth{\labelwidth}{2.n}%
\setlength{\labelsep}{0pt}}}{\end{list}}
\newcounter{Myenumii}
\newenvironment{myenumii}%
{\begin{list}{}{\usecounter{Myenumii}%
\settowidth{\leftmargin}{a)n}\settowidth{\labelwidth}{a)n}%
\setlength{\labelsep}{0pt}}}{\end{list}}
\newcounter{Myenumiii}
\newenvironment{myenumiii}%
{\begin{list}{}{\usecounter{Myenumiii}%
\settowidth{\leftmargin}{iv.n}\settowidth{\labelwidth}{iv.n}%
\setlength{\labelsep}{0pt}}}{\end{list}}

{\end{list}}

\renewenvironment{itemize}%
{\begin{list}{}{%
\settowidth{\leftmargin}{2.n}\settowidth{\labelwidth}{2.n}%
\setlength{\labelsep}{0pt}}}{\end{list}}

\newsymbol\onto 1310

\newcommand{\sA}{{\mathcal A}}

\newcommand{\sF}{{\mathcal F}}

\newcommand{\sH}{{\mathcal H}}

\newcommand{\sJ}{{\mathcal J}}

\newcommand{\sN}{{\mathcal N}}
\newcommand{\sO}{{\mathcal O}}

\newcommand{\sT}{{\mathcal T}}

\newcommand{\sV}{{\mathcal V}}

\newcommand{\A}{{\mathbb A}}

\newcommand{\C}{{\mathbb C}}
\newcommand{\D}{{\mathbb D}}

\newcommand{\F}{{\mathbb F}}

\newcommand{\Q}{{\mathbb Q}}
\newcommand{\R}{{\mathbb R}}

\newcommand{\V}{{\mathbb V}}

\newcommand{\Z}{{\mathbb Z}}

\newcommand{\rank}{{\rm rank}}

\newcommand{\End}{{\rm End}}

\newcommand{\Tr}{{\rm Trace}}

\newcommand{\GL}{{\rm GL}}

\newcommand{\Sp}{{\rm Sp}}

\newcommand{\Hom}{{\rm Hom}}

\newcommand{\Gr}{{\rm Gr}}

\newcommand{\Id}{{\rm Id}}

\newcommand{\Spec}{{\rm Spec}}

\newcommand{\Sym}{{\rm Sym}}

\begin{document}

\title[Hodge Bundles on Smooth Compactifications of Siegel Varieties and Applications]{Hodge Bundles on Smooth Compactifications of \\Siegel Varieties  and Applications}
\author{Shing-Tung Yau}
\address{\rm Department of Mathematics\\
Harvard University\\ Cambridge, MA 02138, USA}
\email{yau@math.harvard.edu}

\author{Yi Zhang}\footnote{Yi Zhang was supported in part by the NSFC Grant(\#10731030) of Key Project "Algebraic Geometry"} %and also was supported partially by LNMS, Fudan University.}
\address{\rm School of Mathematical Sciences \\
Fudan University \\
Shanghai 200433, China}
%\address{Current address: Department of Mathematics\\
%Harvard University\\ Cambridge, MA 02138, USA
%}
\email{zhangyi\_math@fudan.edu.cn, zhangyi.alex@gmail.com}

\date{\today}
\maketitle

{ \small \tableofcontents}

%\vspace{1cm}

Siegel varieties are locally symmetric varieties. They are important and interesting in algebraic geometry and number theory.
We construct a canonical Hodge bundle on a Siegel variety so that the holomorphic tangent bundle can be embedded into the Hodge bundle;
we obtain that the canonical Bergman metric on a Siegel variety  is same as the induced Hodge metric and we describe asymptotic behavior of this unique K\"ahler-Einstein metric
explicitly; depending on these properties and the uniformitarian of K\"ahler-Einstein manifold, we study  extensions of  the tangent bundle over any smooth toroidal compactification
(Theorem \ref{hodge-metric-is-Bergman-metric}, Theorem \ref{degenerate-Bergamn} and Theorem \ref{stable-bundle} in Section 1).
%and study the stability of the logarithmic tangent bundle with respect to the canonical Bergman  current, and obtain that over any smooth toroidal compactification of a Siegel variety
We apply these results of Hodge bundles,  to study dimension of Siegel cusp modular forms and  general type for  Siegel varieties
(Theorem \ref{estimate-dimension} and Theorem \ref{Theorem on general-type-1} in Section 2).\\

%\vspace{1cm}

Throughout this paper, $g$ is an integer more than two.

In this paper, we fix a real vector space $V_\R$  of dimensional $2g$ and fix a standard symplectic form $\psi=\left(
  \begin{array}{cc}
    0 & -I_g \\
    I_g & 0
  \end{array}
\right)$ on $V_\R.$
     For any non-degenerate skew-symmetric bilinear form $\widetilde{\psi}$
on $V_\R,$ it is known that there is an element $T\in \GL(V_\R)$ such that
$^tT\widetilde{\psi}T=\psi.$ We also fix a  symplectic basis
$\{e_{i}\}_{1\leq i\leq 2g}$ of the standard symplectic space
$(V_\R,\psi)$ such that
\begin{equation}\label{standard-symplectic-basis}
    \psi(e_i,e_{g+i})=-1 \mbox{ for } 1\leq i\leq g, \mbox{ and }
  \psi(e_i,e_{j}) =0  \mbox{ for } |j-i|\neq g.
\end{equation}
\begin{itemize}
    \item Denote by $V_\Z:= \oplus_{1\leq i\leq 2g}\Z e_i,$ then
$V_\R=V_\Z\otimes_\Z \R$ and  $V_\Z$ is a standard lattice in
$V_\R.$ In this paper, we fix the lattice $V_\Z$ and fix the
rational space $V_\Q:=V_\Z\otimes_\Z\Q.$
%We write $V$ for $V_\Q$ for convenience,

   \item For any $\Z$-algebra $\frak{R},$ we define $V_{\frak{R}}:=V_\Z\otimes_\Z \frak{R}$
and we write
\begin{equation}\label{Sympletic-group}
\Sp(g, \frak R):=\{h\in \GL(V_{\frak R})\,|\,
\psi(hu,hv)=\psi(u,v)\, \mbox{ for all } u,v \in V_{\frak R}\}.
\end{equation}
%Let $\Sp(g,\Z)$ be the set of $\Z$-points of $\Sp(V,\psi).$
Since $\det
M=\pm 1$ for all $ M\in \Sp(g,\Z),$  $\Sp(g,\Z)$ is a subgroup of
$\Sp(g,\Q).$
\end{itemize}

%\vspace{1cm}

%Section 1

\newpage
%\newpage

\section{Hodge bundles on Siegel varieties}

\vspace{1cm}

Let $\Gamma$ be a neat arithmetic subgroup of $\Sp(g,\Q).$ Let $V_\Q$ be the fixed rational symplectic vector space as in the introduction of notations.

By Borel's embedding theorem, the Siegel space $\frak H_g$ can be realized as a bounded domain parameterizing weight one polarized Hodge structures(cf.Proposition \ref{Borel-embedding} in A2). Moreover,
there is a natural variation of Hodge structures on the  Siegel space $\frak H_g$ :
\begin{mycor1}[Cf.\cite{Del79}]
Gluing Hodge structures on $\frak{H}_g$
altogether, the local system $\V:=V_\Q\times \frak{H}_g$  admits a
homogenous variation of polarized rational Hodge structures of
weight one on $\frak{H}_g.$
\end{mycor1}

Let $o$ be an arbitrary fixed base point in $\sA_{g,\Gamma}.$
Since $\frak{H}_g$ is simply connected, the fundamental group of $\sA_{g,\Gamma}$
has
$\pi_1(\sA_{g,\Gamma}, o)\cong\Gamma.$ Then, there is  a natural local system
$\V_{g,\Gamma}:=V_\Q\times_{\Gamma}\frak{H}_g$
 on $\sA_{g,\Gamma}$ given by the fundamental representation
$\rho:\pi_1(\sA_{g,\Gamma}, o)\to \mathrm{GSp}(V,\psi)(\Q).$
Actually the $\Q$-local system $\V_{g,\Gamma}$ admits a variation of polarized rational Hodge structures of weight one
on $\sA_{g,\Gamma}:=\Gamma\backslash\frak{H}_g$ by using the arguments
in Section $4$ of \cite{Zuc81}.  More precisely, in our previous paper (cf Proposition 1.8 in Section 1 of \cite{Y-Z}) we obtain :
\begin{myenumi}
\item The local system  $\V_{g,\Gamma}$ admits a variation of
polarized rational Hodge structures on $\sA_{\Gamma},$ and the
associated period map attached to this PVHS is
\begin{equation}\label{period-mapping}
 \iota_{\Gamma}: \sA_{g,\Gamma}\>\cong>> \Gamma \backslash\frak{S}_g.
\end{equation}

\item Let  $\widetilde{\sA}_{g,\Gamma}$ be an arbitrary smooth
compactification of $\sA_{g,\Gamma}$ with simple normal crossing
divisor
 $D_\infty:=\widetilde{\sA}_{g,\Gamma}\setminus\sA_{g,\Gamma}.$ Around the boundary divisor
 $D_\infty,$ all local monodromies of any rational PVHS $\widetilde{V}$ on $\sA_{g,\Gamma}$  are unipotent.\\
\end{myenumi}

Now,  we fix this rational PVHS $\V_{g,\Gamma}$ throughout this paper.

\subsection{Construction of Hodge bundles on Siegel varieties}
Most materials in this subsection are taken from \cite{Sch73},\cite{Sim90} and \cite{Zuo00}.

We note that $\mathbb{H}:=\V_{g,\Gamma}\otimes \C$ is a flat complex vector bundle on the $\sA_{g,\Gamma}$ with a flat connection  $\D.$ %Actually $\mathbb{H}$ admits a polarized rational VHS $R^1\chi_{g,n*}(\Q),$
There is a filtration of $C^\infty$ vector bundles over
$\sA_{g,\Gamma}$
$0=\F^2\subset\F^1\subset \F^0=\mathbb{H},$  whose
fibers at each point $\tau\in \sA_{g,\Gamma}$ gives a Hodge
filtration isomorphic to $F_\tau^\bullet:=(0\subset F^1_\tau\subset V_\C).$
%$$0\subset H^{1,0}_\tau \subset \mathbb{H}_\tau=V_\C.$$
%$$0\subset H^1(\frak{X}_{\tau}, \Omega^1_{\frak{X}_{\tau}})\subset H^1(\frak{X}_{\tau},\C)=\mathbb{H}_\tau.$$
The vector bundle $\mathbb{H}$ admits a
positive Hermitian metric $H$ induced by the polarization $\psi$ of the Hodge structures as
follows:
\begin{equation}\label{Hodge-metric}
    <u,\overline{v}>_{H}:=\psi(C_\tau u,\overline{v})\,\, \forall  u,v\in \mathbb{H}_\tau,
\end{equation}
where each $C_\tau$ is the Weil operator on the $F_\tau^\bullet.$ We usually call this metric
$H$ the \textbf{Hodge metric} on $\mathbb{H}.$
Let $\mathbb{H}^{p,q}:= \F^p\cap \overline{\F^{q}}.$  The smooth
decomposition
$\mathbb{H}=\bigoplus \mathbb{H}^{p,q} $
is orthogonal with respect to the Hodge metric $H.$

Let $\D^{0,1}$ be the $(0,1)$-part of the flat connection $\D$
and $\D^{1,0}$ the $(1,0)$-part of $\D.$ The $\D^{0,1}$
  gives a holomorphic structure on
$\mathbb{H},$ so that $\sH:=(\mathbb{H}, \D^{0,1})$ is the corresponding holomorphic bundle.
 The $\D^{1,0}$ guarantees $\sH$ has an integrable holomorphic connection
$\D^{1,0} :  \sH \to \sH\otimes \Omega_{\sA_{g,\Gamma}}^1.$
It is known that  all subbundles $\F^p$'s admit the holomorphic structure
$\D^{0,1}$ naturally, so that we have the corresponding holomorphic subbundles
$\sF^p.$ Moreover, we have  the Griffiths transversality :
\begin{equation}\label{Griffiths-transversality}
    \D^{1,0}: \sF^p\>>>\sF^{p-1} \otimes \Omega_{\sA_{g,\Gamma}}^1,
\forall p.
\end{equation}

Define
$E^{p,2-p}:=\sF^p/\sF^{p+1}\,\,\,\forall p.$
We know that each holomorphic vector bundle $E^{p,q}$ is  $C^\infty$-isomorphic to the vector bundle $\mathbb{H}^{p,q}.$
%Actually, we have $$E^{p,q}=R^q\pi_*(\Omega_{\frak{X}_{g,n}/\sA_{g,\Gamma}}^p), \,\, p+q=1.$$
We set $E:=\Gr(\sH)=\bigoplus_p E^{p,n-p}.$
The  flat connection $\D$ on $\sH$ actually induces a global  holomorphic structure
$\overline{\partial}$ on $E$ such that each $E^{p,q}$ is a holomorphic subbundle od $E.$
We write :
\begin{equation}\label{preHodge-bundles}
   E^{p,q}=(\mathbb{H}^{p,q},\overline{\partial}),\mbox{ and } E=(\bigoplus \mathbb{H}^{p,q},\overline{\partial}).
\end{equation}
The holomorphic vector bundles $E$ and $E^{p,q}$'s  are endowed natural Hermitian metrics induced by $H.$
For convenience, we still call these Hermitian metrics the Hodge metrics and still write these Hermitian metrics as $H.$ \\

Let $\mathbb{T}(\sA_{g,\Gamma})$ be the real tangent bundle  of $\sA_{g,\Gamma}.$
According to $\pm\sqrt{-1}$-eigenvalues of the complex structure
$J$ on $\mathbb{T}(\sA_{g,\Gamma}),$ there is a $C^\infty$
decomposition
$\mathbb{T}(\sA_{g,\Gamma})\otimes \C=\mathbb{T}^{1,0}(\sA_{g,\Gamma})\oplus \mathbb{T}^{0,1}(\sA_{g,\Gamma}).$
The real tangent bundle
$\mathbb{T}(\sA_{g,\Gamma})$
undertakes the holomorphic tangent bundle
$\sT_{\sA_{g,\Gamma}}:=(\Omega_{\sA_{g,\Gamma}}^1)^\vee$ in sense
that
$$\mathbb{T}^{1,0}(\sA_{g,\Gamma})\>\cong>C^\infty>\sT_{\sA_{g,\Gamma}},
\,\,\mathbb{T}^{0,1}(\sA_{g,\Gamma})\>\cong>C^\infty>\overline{\sT_{\sA_{g,\Gamma}}}.$$
Let $(p,q)$ be a pair of integers. For any local holomorphic vector filed $\overrightarrow{X}$ of $
T_{\sA_{g,\Gamma}},$ there is a local
$\sO_{\sA_{g,\Gamma}}$-linear morphism
$\theta^{p,q}(\overrightarrow{X}) :  E^{p,q}\to E^{p-1,q+1}$
by the Griffiths transversality \ref{Griffiths-transversality}. Then we get an
$\sO_{\sA_{g,\Gamma}}$-linear morphism
$\theta^{p,q}:E^{p,q}\to E^{p-1,q+1} \otimes\Omega^{1}_{\sA_{g,\Gamma}},$
and so we get
the adjoint map $\theta^{p-1,q+1*}_H: E^{p-1,q+1}\to E^{p,q}\otimes
\overline{\Omega^1_{\sA_{g,\Gamma}}}$ of $\theta^{p,q}$ given by
$<\theta^{p,q}(s),\overline{t}>_H=<s,\overline{\theta^{p-1,q+1*}_H(t)}>_H,$ where $s$(resp. $t$) is a local section of $E^{p,q}$(resp. $E^{p-1,q+1}$). Clearly $\theta^{p,q}$ can be regarded as  an $\overline{\sO_{\sA_{g,\Gamma}}}$-linear morphism.
The \textbf{Higgs field} $\theta$ on $E$ is defined as follows :
$$
\theta=\bigoplus_{p.q}\theta^{p,q} :\bigoplus_{p,q} E^{p,q}\>>>\bigoplus_{p,q}
E^{p,q}\otimes \Omega^1_{\sA_{g,\Gamma}}.
$$
Respectively, the adjoint morphism of $\theta$ is defined to be $\theta^*_H :=\bigoplus\limits_{p,q} \theta^{p,q*}_H.$
\begin{myrem}Let $A^{1}$ be the dual of the sheaf of $C^\infty$ germs of $\mathbb{T}(\sA_{g,\Gamma}).$ Then there is a $C^\infty$ splitting $A^1=A^{1,0}\oplus A^{0,1}$ where $A^{1,0}$(resp. $A^{0,1}$) is the dual of the sheaf
of $C^\infty$ germs of $\sT_{\sA_{g,\Gamma}}$(resp.
$\overline{\sT_{\sA_{g,\Gamma}}}$).
We can extend $\theta$ and $\theta^*_H$  naturally as $C^\infty$
morphisms
\begin{eqnarray*}
  \theta : C^\infty(E)  &\>>>& C^\infty(E)\otimes A^{1,0}, \\
  \theta^*_H : C^\infty(E)  &\>>>& C^\infty(E)\otimes A^{0,1},
\end{eqnarray*}
where $C^\infty(E)$ is the sheaf of $C^\infty$ germs of $E.$
\end{myrem}

Let $\nabla_H$ be the unique Chern connection on $(E,H).$ Thus,
the connection $\nabla_H$ is compatible with the Hodge metric, and its $(0,1)$-part has
$\nabla_H^{0,1}=\overline{\partial}.$
Define $\partial:= \nabla_H^{1,0}.$
We immediately obtain
$\partial^2 = \overline{\partial}^2=0,$
and get the Chern curvature form
$$\Theta(E,H):=\nabla_H\circ \nabla_H=(\nabla_H^2)^{1,1}.$$
\begin{lemma}\label{holomorphic-Higgs-field}
We have :
\begin{eqnarray*}
  \overline{\partial}(\theta) &:=& \overline{\partial}\circ \theta+ \theta\circ \overline{\partial}=0, \\
  \partial(\theta^*_h) &:=& \partial\circ \theta^*_h +\theta^*_h\circ\partial=0.
\end{eqnarray*}
\end{lemma}
\begin{proof}
One can find these two equalities in \cite{Sim88}\&\cite{Sim90}.
Here we give a direct proof.

The morphism $\theta$ is naturally holomorphic by the definition,so that
the first equality is automatically true. Now, we begin to prove the second
equality.

It is sufficient to prove the equality at an arbitrary point $p.$
Let $(U,p)$ be a local coordinate neighborhood of $p.$ Let
$\{e_\alpha\}$ be a local holomorphic basis of $E.$ We then get a local holomorphic basis $\{e^\alpha\}$  of $E^{\vee}|_U:=\Hom(E|_U,\sO_U)$ as follows : For each
$\alpha,$ let $e^\alpha$ be the dual of $e_\alpha,$ i.e.,
$e^\alpha\in E^\vee|_{U}$ such that
$e^\alpha(e_\beta)=\left\{%
\begin{array}{ll}
    1, & \beta=\alpha; \\
    0, & \beta\neq\alpha.
\end{array}%
\right. $
We call the  local holomorphic basis $\{e^\alpha\}$ of $E^{\vee}|_U$ as a local dual base of $\{e_\alpha\}.$
Let $\{l_1,\cdots, l_m\}$ be a local holomorphic basis of
$\sT_{\sA_{g,\Gamma}}$ and  $\{\phi_1,\cdots, \phi_m \}\subset
\Omega^1_{\sA_{g,\Gamma}}$ its local holomorphic dual basis.
Locally, we can write
$$\theta=\sum_{i=1}^m A^i\phi_i,  \,\,\,\,\theta^*_h=\sum_i B^i\overline{\phi_i}$$
where $ A^i:=A^{i,\alpha}_\beta e_\alpha\otimes e^\beta$ and
$$
B^i:=B^{i,\alpha}_\beta e_\alpha\otimes e^\beta \mbox{ with } B^{i,\beta}_\alpha:=\sum_{\gamma,\delta}H_{\alpha\bar{\gamma}}\overline{A^{i,\gamma}_\delta}H^{\bar{\delta}\beta},
H_{\alpha\bar{\gamma}}:=<e_\alpha,\overline{e_\gamma}>_H.$$
%where $A^{i,\gamma}_\alpha\in\sO_U,\overline{B^{i,\gamma}_\alpha}\in\sO_U. $
Form the first equality in the lemma, we get
\begin{equation}\label{calculation-2}
    0=\overline{\partial}\theta=\sum_{i=1}^m\sum_{j=1}^n A^i_{;\bar{j}} \bar{\phi_j} \wedge
    \phi_i=\sum_{i=1}^m\sum_{j=1}^nA^{i,\alpha}_{\beta ; \bar{j}} e_\alpha\otimes
e^\beta\bar{\phi_j} \wedge
    \phi_i
\end{equation}
where $A^{i,}_{;\bar{j}}$'s for all $j$ are covariant partial
derivations of the tensor $A^i,$ and so we obtain
$$A^{i,\alpha}_{\beta ; \bar{j}}=0 \,\,\mbox{  on  } \,\,U \,\,\forall i, j, \alpha,\beta.$$
On the other hand, we compute that
\begin{equation}\label{calculation-1}
    \partial\theta_h^*=\sum_{i=1}^m\sum_{j=1}^m B^i_{;j} \phi_j \wedge
    \overline{\phi_i}=\sum_{i=1}^m\sum_{j=1}^mB^{i,\alpha}_{\beta ; j} e_\alpha\otimes
e^\beta \phi_j \wedge
    \overline{\phi_i},
\end{equation}
where $B^{i,}_{;j}$'s for all $j$ are  covariant partial derivations
of the tensor $B^i.$ It is well-known that one can contract the neighborhood $(U,p)$ sufficiently small to get a
special holomorphic local basis $\{e_\alpha\}$ of $E$ over $U$
such that
$H(p)=\Id, \,\, dH(p)=0$
under the frame $\{e_\alpha\}.$ Then, at the point $p,$ We have :
$$B^{i,\alpha}_{\beta ; j}(p)= \overline{A^{i,\alpha}_{\beta ; \bar{j}}(p)}=0,\,\,\forall i, j \alpha,\beta.$$
Thus $\partial\theta_h^*=0$ at the point $p.$
%Therefore we obtain the second equality in the lemma.
\end{proof}

\begin{corollary}\label{Hodge-metric-curvature}
Let $(E,H)$ be Hermitian vector bundle in \ref{preHodge-bundles}.
We have :
\begin{eqnarray*}
  \Theta(E,H) &=& -(\theta\wedge\theta^*_h+\theta^*_H\wedge\theta),\\
  \theta\wedge \theta&=&-\partial(\theta)=0,\\
  \theta^*_H\wedge\theta^*_H&=&-\overline{\partial}(\theta^*_H)=0.
\end{eqnarray*}
\end{corollary}
\begin{proof}
It is known the flat connection $\D$ on $\mathbb{H}$ has the following
decomposition
$$\mathbb{D}=\nabla_H +\theta+ \theta^*_H.$$
Since $\mathbb{D}^2=0,$ we can finish the proof by the lemma \ref{holomorphic-Higgs-field}.\\
\end{proof}
Attached to the PVHS $\V_{g,\Gamma},$ we finally obtain the associated
\textbf{Hodge bundle} $(E,\theta,H)$ on $\sA_{g,\Gamma},$ i.e., a
holomorphic system
\begin{equation}\label{Hodge-bundle}
(E=\oplus E^{p,q},\theta=\oplus \theta^{p,q})
\end{equation}
with a Hermitian metric $H$ satisfying the following properties :
\begin{itemize}
  \item $E^{p,q}$ are orthogonal to each other under the metric $H;$
  \item $\theta\wedge \theta=0;$
  \item $\theta^{p,q} : E^{p,q} \>>>  E^{p-1,q+1}\otimes\Omega^1_{\sA_{g,\Gamma}}.$\\
\end{itemize}

The dual local system $\V_{g,\Gamma}^{\vee}=V_\Q^{\vee}\times_{\Gamma}\frak{H}_g$ admits a polarized rational VHS of
weight $-1$ on $\sA_{g,\Gamma},$  its associated Hodge
bundles is $(E^{\vee}=\bigoplus_{p+q=1}E^{\vee -p, -q},
\theta_{\vee})$ with
\begin{eqnarray*}\label{dual-Higgs-bundle}
  E^{\vee -p, -q} &=& (E^{p, q})^{\vee}=E^{q, p}, \\
  \theta_{\vee}^{-p, -q} &=&-\theta^{q, p}: E^{\vee -p, -q}\>>>E^{\vee -p-1, -q+1}\otimes
\Omega^1_{\sA_{g,\Gamma}}.
\end{eqnarray*}

Similarly, the local system  $\End(\V):=\mathrm{End}(V_\Q)\times_{\Gamma}\frak{H}_g$
 admits a polarized rational VHS of weight $0$ on $\sA_{g,\Gamma},$ its
associated Hodge bundle is
$(\End(E),\theta^{end})$ with
$$ \End(E)= \bigoplus_{(p-p')+(q-q')=0}E^{p, q}\otimes E^{\vee -p',-q'}$$
and the Higgs field $\theta^{end}:\End(E)\to \End(E)\otimes
\Omega^1_{\sA_{g,\Gamma}}$ given by
$$ \theta^{end}(u\otimes v^{\vee})=\theta(v)\otimes
v^{\vee}+u\otimes \theta_{\vee}(v^{\vee}).$$
We notes that $\End(E)$ has a holomorphic
subbundle
\begin{eqnarray*}
% \nonumber to remove numbering (before each equation)
   \End(E)^{-1,1}&=&  \bigoplus_{p+q=1}E^{p,q}\otimes E^{q-1,p+1},\\
                 &=&  (E^{0,1})^{\otimes 2}. %=(R^1\pi_*(\sO_{\frak{X}_{g,n}}))^{\otimes 2}
\end{eqnarray*}

We still use $H$ to denote the induced Hermitian metric on
$E^\vee$ and $\End(E).$ Throughout this  section, we now fix the Hermitian bundles $(E,H),$ $(\End(E),H),$ and $(E^{pq},H)$'s, $(\End(E)^{p,q},H)$'s.\\

\subsection{Degeneration of canonical metrics on Siegel varieties}
Let $\widetilde{\sA}_{g,\Gamma}$ be a smooth compactification
of $\sA_{g,\Gamma}$ such that the
divisor  $D_{\infty}=\widetilde{\sA}_{g,\Gamma}\setminus\sA_{g,\Gamma}$ is  simple normal crossing. % $D_{\infty}=\cup_i D_i$
 Since any local
monodromy  of  $\V_{g,\Gamma}$ around $D_{\infty}$ is
unipotent, the Hodge bundle $(E,\theta)$ has a
Deligne's canonical extension
$(\overline{E}=\oplus\overline{E^{p,q}},\overline{\theta}=\oplus
\overline{\theta^{p,q}})$ with
$\overline{\theta^{p,q}}: \overline{E^{p,q}}\to\overline{E^{p-1,q+1}}\otimes
\Omega^1_{\widetilde{\sA}_{g,\Gamma}}(\log D_{\infty}).$  Deligne's extension of $(End(E),\theta^{end})$ is
$(End(\overline{E}),\overline{\theta^{end})}.$
The morphism
$
    \overline{\theta^{1,0}}: \overline{E^{1,0}}\to \overline{E^{0,1}}\otimes \Omega^1_{\widetilde{\sA}_{g,\Gamma}}(\log D_{\infty})
$
represents the  global section
$\overline{\theta^{1,0}}\in H^0(\widetilde{\sA}_{g,\Gamma}, \overline{E^{0,1}}^{\otimes 2}\otimes\Omega^1_{\widetilde{\sA}_{g,\Gamma}}(\log D_{\infty})).$
Then, we obtain a sheaf morphism
\begin{equation}\label{kodaira-Spencer-class}
    \rho:  \sT_{\widetilde{\sA}_{g,\Gamma}}(-\log D_{\infty})\>>> \overline{E^{0,1}}^{\otimes 2}.
\end{equation}
Define the restriction map $\rho_0:=\rho|_{\sA_{g,\Gamma}}.$
\begin{lemma}\label{subbundle-tangent}
The holomorphic tangent bundle $\sT_{\sA_{g,\Gamma}}$ of $\sA_{g,\Gamma}$ is a holomorphic subbundle of $(E^{0,1})^{\otimes 2}.$
Moreover, the morphism
$\rho_0:  \sT_{\sA_{g,\Gamma}}\>>> (E^{0,1})^{\otimes 2}$
is an inclusion of vector bundles.
\end{lemma}
\begin{proof}
We know that the vector bundles $E^{1,0},$ $E^{0,1},$ $\sO_{\sA_{g,\Gamma}},$ $\sT_{\sA_{g,\Gamma}},$ $\Omega^1_{\sA_{g,\Gamma}}$ are
all $\Sp(g,\R)$-homogenous, and the morphism
$\theta^{1,0}: E^{1,0}\>>> E^{0,1}\otimes \Omega^1_{\sA_{g,\Gamma}}$
is a $\Sp(g,\R)$-equivariant morphism. Thus, the morphism
$\rho_0:  \sT_{\sA_{g,\Gamma}}\>>> (E^{0,1})^{\otimes 2}$ is $\Sp(g,\R)$-equivariant.
We  verify the inclusion at
the base point $o\in\sA_{g,\Gamma}$ : At point $o,$ we have
$E^{1,0}|_{o}=H^{1,0}_o, E^{0,1}|_{o}=H^{0,1}_o$ and
$\sT_{\widetilde{\sA}_{g,\Gamma},o}\subset
\mathrm{Hom}(H^{1,0}_o,H^{0,1}_o)=(H^{0,1}_o)^{\otimes 2}
$ by Borel's embedding.
The construction of the Hodge bundle $(E,\theta)$ shows that the inclusion  $\sT_{\widetilde{\sA}_{g,\Gamma},o}\subset
\mathrm{Hom}(H^{1,0}_o,H^{0,1}_o)=(H^{0,1}_o)^{\otimes 2}
$ is just the morphism $\rho_0$ at the point $o.$
\end{proof}

%\begin{myrem}
%In fact, the isomorphism of $h_{g,n}$ in the proposition \ref{vhs
%on Siegel varieties} means the injectivity and subjectivity of the
%global Torelli map attached to the PVHS $R^1\chi_{g,n *}(\Q).$ The
%restriction of $\rho$ to $\sA_{g,\Gamma}$ is just the differential $d
%h_\Gamma:
%\sT_{\sA_{g,\Gamma}}\>\hookrightarrow>>(R^1\pi_*(\sO_{\sX_{g,n}}))^{\otimes
%2}, $ which in geometry is compatible with the Kodaira-Spencer
%class  of the universal family $\chi_{g,n}.$\\
%\end{myrem}

We now introduce an induced $\Sp(g,\R)$-invariant positive Hermitian metric $H$(\textbf{Hodge
metric})  on  $\sA_{g,\Gamma}$ by the following inclusion $$\rho_0 :
\sT_{\sA_{g,\Gamma}}\>\subset>>\End(E)^{-1,1}\subset\End(E).$$
Let $\{l_1,\cdots, l_m\}$ be a
holomorphic basis of $\sT_{\sA_{g,\Gamma}}$ on a local
neighborhood $(U,z)$ of $\sA_{g,\Gamma},$ and $\{\phi_1,\cdots,
\phi_m\}$ be the dual holomorphic basis of
$\Omega^1_{\sA_{g,\Gamma}}$ over $U.$ We define
\begin{equation}\label{Hodge-metric-slice}
    H(l_i, \overline{l_j}):=<\rho_0(l_i), \overline{\rho_0(l_j)} >_H.
\end{equation}
Since  $\rho_0$ can be
linearly extended to a morphism of sheaves of $C^\infty$ germs
as well as $\theta$ does, we then obtain a metric $H$ on $\sA_{g,\Gamma}.$
The K\"ahler form of $H$ on $U$ can be written locally as
\begin{equation}\label{Kaihler-form-Hodge-metric}
    \omega_H=\sum_{i,j=1}^m H(l_i,\overline{l_j}) \phi_i\wedge\overline{\phi_j}.
\end{equation}

\begin{theorem}\label{hodge-metric-is-Bergman-metric}
Let $\Gamma$ be a neat arithmetic subgroup of $\Sp(g,\Z).$

The induced Hodge metric $H$ on the Siegel variety
$\sA_{g,\Gamma}=\Gamma\backslash\frak{H}_g$ is same as the
canonical Bergman metric. Moreover,the Chern connection of
$(\sT_{\sA_{g,\Gamma}},H)$ is compatible with the Levi-Civita
connection of the Riemannian manifold $(\sA_{g,\Gamma},H).$
\end{theorem}
\begin{proof}Notation as in the proof of the lemma
\ref{holomorphic-Higgs-field}.
Since the Hodge metric $H$ on $\sA_{g,\Gamma}$ is
$\Sp(g,\R)$-invariant and $\Sp(g,\R)$ is a simple group, it is
sufficient to show that $H$ is K\"ahler.

Let $p$ be an arbitrary point on $\sA_{g,\Gamma}.$ Let $U$ be a
suitable neighborhood of $p$ such that we can choose a local
holomorphic coordinates $(z_1,\cdots,z_m)$ satisfying
$$\sT_{\sA_{g,\Gamma}}|_{U}=\mathrm{span}\{\frac{\partial}{\partial z_1},\cdots, \frac{\partial}{\partial z_m}\}.$$
Let $\{e_\alpha\}$ be a local holomorphic basis of $E$ and $\{e^\alpha\}$  the local dual holomorphic basis of $E^{\vee}$.

All calculation below are locally over $U.$

We write
$\theta=\sum_{i=1}^k A^id z_i,$
where $A^i=\sum_{\alpha,\beta}A^{i,\alpha}_\beta e_\alpha\otimes
e^\beta\in \End(E).$ The K\"ahler form is then
$$\omega_H:=\sum_{i,j=1}^k H(l_i,\overline{l_j})dz_i\wedge\overline{dz_j}=\sum_{i,j=1}^k<A^i, \overline{A^j}>_H dz_i\wedge\overline{dz_j}.$$

Thus, we have that
\begin{eqnarray*}
  d \omega_H &=& \sum_{i,j}(d<A^i, \overline{A^j}>_H)\wedge dz_i\wedge\overline{dz_j} \\
   &=& \sum_{i,j=1}^m(<\nabla_H A^i,\overline{A^j} >_H + <A^i, \overline{\nabla_H
   A^j}>)\wedge dz_i\wedge\overline{dz_j},
  \end{eqnarray*}
where $\nabla_H$ is the chern connection on $(\End(E),H).$
For each $i=1, \cdots, m,$  We have :
\begin{eqnarray*}
% \nonumber to remove numbering (before each equation)
  \nabla_HA^i &=& \overline{\partial}A^i+\partial A^i=\partial A^i  \\
   &=& \sum_{k=1}^n A^{i,\alpha}_{\beta;k} e_\alpha\otimes e^\beta \phi_k.
\end{eqnarray*}
Since  $\partial(\theta) =0$ by the corollary
\ref{Hodge-metric-curvature}, we have
\begin{equation}\label{covariant-derivation}
    A^{i,\delta}_{\alpha; j}= A^{j,\delta}_{\alpha; i} \,\, \forall
    i,j,\alpha, \delta.
\end{equation}
Contract the neighborhood $(U,p)$ sufficiently small, we can choose a
special holomorphic basis $\{e_\alpha\}$ of $E$ over $U$ such that
$H(p)=\Id, \,\, dH(p)=0$ under the frame $\{e_\alpha\}.$
%Then, $$A^{i,\alpha}_\beta(p)=A^{i,\beta}_\alpha(p), \forall
%i,\alpha,\beta.$$
At the point $p,$ we calculate
\begin{eqnarray*}
  (d^{1,0}\omega_H)(p)&=& \sum_{i,j,l}\sum_{\alpha,\beta}< A^{i,\alpha}_{\delta;l}e_\alpha\otimes e^\delta, \overline{A^{j,\tau}_\beta e_\tau\otimes e^\beta }>_H dz_l\wedge dz_i\wedge \overline{dz_j}\\
   &=& \sum_{l,i,j=1}^m\sum_{\alpha,\beta}A^{i,\alpha}_{\beta;l}\overline{A^{j,\alpha}_\beta} dz_l\wedge dz_i\wedge \overline{dz_j}\\
   &=& \sum_{j=1}^m\sum_{\alpha,\beta}\overline{A^{j,\alpha}_\beta} (\sum_{i,l=1}^kA^{i,\alpha}_{\beta;l}dz_l\wedge dz_i)\wedge \overline{dz_j} \\
   &=& 0.
\end{eqnarray*}
Similarly, $d^{0,1}\omega_H=0$ at the point $p.$

The rest is obvious.
\end{proof}

We still use $H$ to represent the dual metric of
$\Omega_{\sA_{g,\Gamma}}^1$ induced by $(\sT_{\sA_{g,\Gamma}}, H).$ We
write $\Theta(\sT_{\sA_{g,\Gamma}},H)$ (resp.
$\Theta(\Omega^1_{\sA_{g,\Gamma}},H)$) as the Chern curvature form
of the vector bundle $\sT_{\sA_{g,\Gamma}}$ (resp.
$\Omega^1_{\sA_{g,\Gamma}}$). As the canonical Bergman metric on
$\sA_{g,\Gamma}$ is K\"ahler-Einstein, there is
$\frac{-1}{2\pi\sqrt{-1}}\Tr_H(\Theta(\sT_{\sA_{g,\Gamma}},H))=-\lambda\omega_{H},$
where $\lambda$ is a positive constant. Without lost of
generality, we always assume $\lambda=1$ for convenience.

\begin{theorem}\label{degenerate-Bergamn}
Let $\Gamma$ be a neat arithmetic subgroup of $\Sp(g,\Z).$ Let
$\widetilde{\sA}_{g,\Gamma}$ be an arbitrary smooth
compactification(not necessary smooth toroidal compactification) of the Siegel variety $\sA_{g,\Gamma}:=
\Gamma\backslash\frak{H}_g$ such that
$D_{\infty}=\widetilde{\sA}_{g,\Gamma}\setminus\sA_{g,\Gamma}$ is a simple
normal crossing divisor. We have :
\begin{myenumi}
\item The canonical Bergman metric $H_{\mathrm{can}}$ of $\sA_{g,\Gamma}$
is bounded by the logarithmic degeneration  along the boundary divisor $D_{\infty}$ in sense of the following description :\\

\noindent Let $p$ be a point in $\widetilde{\sA}_{g,\Gamma}$ with a coordinate chart $(U,(z_1,\cdots, z_n))$($n=g(g+1)/2$) such that
$$U\cap \sA_{g,\Gamma}=\{(z_1,\cdots,z_l,\cdots z_n) \,\,| \,\,   0<|z_i|<1(i=1,\cdots l),\, |z_j|<1(i=l+1,\cdots,n) \,\,\}.$$
 Let $\omega_{\mathrm{can}}$ be the K\"ahler form of the canonical Bergman metric $H_{\mathrm{can}}.$ There holds
$$\frac{1}{C}(\prod_{i=1}^l-\log |z_i|)^{M}\leq |\omega_{\mathrm{can}}|\leq C(\prod_{i=1}^l-\log |z_i|)^{M}$$
in the coordinate chart $\{(z_1,\cdots,z_l,\cdots, z_n)\,\,| 0<|z_i|<r(i=1,\cdots,l), |z_j|<r(j=l+1,\cdots,n) \}$ of $p$
for a suitable $r>0,$ where $C,M$ are positive constants depending on $r.$

%In particular, the form $\omega_h$  has Poincar\'e growth on $D.$

\item The K\"ahler form $\omega_{\mathrm{can}}$ becomes a closed positive current $[\omega_{\mathrm{can}}]$ on $\widetilde{\sA}_{g,\Gamma}.$

\item The line bundle $$\omega_{\widetilde{\sA}_{g,\Gamma}}(D_{\infty})=\bigwedge^{\dim \sA_{g,\Gamma}}\Omega^1_{\widetilde{\sA}_{g,\Gamma}}(\log D_{\infty}) $$
is pseudo-effective on
$\widetilde{\sA}_{g,\Gamma}.$ Precisely, there is an equality
$c_1(\Omega^1_{\widetilde{\sA}_{g,\Gamma}}(D_{\infty}))=[\omega_{\mathrm{can}}].$

\item The line bundle $\omega_{\widetilde{\sA}_{g,\Gamma}}(D_{\infty})$ is  big on $\widetilde{\sA}_{g,\Gamma}.$
\end{myenumi}
\end{theorem}
\begin{myrem}
It is well known that Bergman metric on locally symmetric has Poincar\'e growth on $D_\infty.$
Our result is strong than  this classic result in literature.\\
\end{myrem}

Before proving the theorem \ref{degenerate-Bergamn}, we review the theory of degeneration
of Hodge metrics  on  any polarized variation of Hodge structures over a quasi K\"ahler  manifold.

Let $X$ be an open K\"ahler manifold of complex dimension $m.$ Let $\overline{X}$ be one smooth compactification of $X$ such that the boundary $D:=\overline{X}-X$ is a simple normal crossing divisor. Let $j:X\>\subset>> \overline{X}$ be the open embedding. Let $\mathbb{V}$ be an arbitrary polarized variation of real Hodge
structures  over $X$ such that all monodromies
around $D$ are unipotent. Denote by $\sV=\mathbb{V}\otimes \sO_X.$ Consequently, we have
a Hodge filtration
  $$\sV=\sF^0\supset\sF^1\supset\cdots\supset\sF^w\supset 0$$
corresponding to the VHS $\V,$ where $w$ is  the weight of the VHS $\mathbb{V}.$

Let $(\triangle_1,z)\subset \overline{X}$ be a \textbf{special coordinate neighborhood}, i.e.,
a  coordinate neighborhood  isomorphic to the polycylinder $\Delta^m$ ($\Delta:= \{z\in\C\,|\,   |z|\leq 1 \}$) such that
$$X\cap \triangle_1\cong\{z=(z_1,\cdots,z_l,\cdots, z_{m})\in \Delta^m \,\,|\,\, z_1\neq 0,\cdots, z_l\neq 0\}= (\Delta^*)^l\times \Delta^{m-l}.$$
We then have $\triangle_1\cap D_\infty\cong \{(z_1,\cdots,
z_l,\cdots, z_{m})\,\,|\,\, z_1\cdots z_l=0\}.$
For any real number $0<\varepsilon<1,$ let $\triangle_\varepsilon$ be a scaling neighborhood of $\overline{X},$ i.e.,
$$\triangle_\varepsilon\subset \triangle_1 \mbox{ and } \triangle_\varepsilon \cong \{ (z_1,\cdots, z_l,\cdots, z_m)\in \Delta^m\,\, | \,\, |z_\alpha|\leq \varepsilon \mbox{ for } \alpha=1,\cdots, l\}.$$

Let $\gamma_\alpha$ be a local monodromy around $z_\alpha=0$ in
$\triangle_1$ for $\alpha=1, \cdots l.$ Denote by $$N_\alpha=\log
\gamma_\alpha:= \sum_{j\geq 1}(-1)^{j+1}\frac{(\gamma_\alpha-1)^{j}}{j},\,\, \forall \alpha,$$ then each $N_\alpha$ is nilpotent. Let $(v_{\cdot})$ be a flat multivalued basis  of $\sV$ over
$\triangle_1\cap X.$  The formula
$$(\widetilde{v_{\cdot}})(z):=\exp(\frac{-1}{2\pi\sqrt{-1}}\sum_{\alpha=1}^l\log z_\alpha N_\alpha )(v_{\cdot})(z)$$
gives a single-valued basis of $\sV.$ Deligne's canonical
extension $\overline{\sV}$ of $\sV$ to $\triangle_1$ is generated
by $(\widetilde{v_\cdot})$ (cf.\cite{Sch73}). The
construction of $\overline{\sV}$ is independent of the choice of
$z_i's$ and $(v_\cdot).$
For any holomorphic subbundle $\sN$ of $\sV,$ Deligne's extension of $\sN$ is defined to be $\overline{\sN}:=\overline{\sV}\cap j_*\sN.$ Then, we have extension of the filtration $$\overline{\sV}=\overline{\sF}^0\supset\overline{\sF}^1\supset\cdots\supset\overline{\sF}^w\supset
0,$$ which is also a filtration of locally free sheaves.

%Locally, we can choose multivalued basis $(v_{\cdot})$ of $\sV$ as follows
%$$(v_{\cdot})=(v_{w,1},\cdots v_{w, n_{w}}, v_{w-1,1}, \cdots,v_{w-1,n_{w-1}}, \cdots, v_{0,1},\cdots, v_{0, n_0} )$$
%such that $(v_{w,1},\cdots v_{w, n_{w}}, \cdots,v_{\alpha,1}\cdots,v_{\alpha,n_{\alpha}})$ is a multivalued basis of
%$\sF^\alpha.$

Let $N$ be a linear combination of $N_\alpha.$
Then $N$ defines a weight flat filtration $W_{\bullet}(N)$ of $\V_\C$ (\cite{Del71},\cite{Sch73}) by
$$0\subset \cdots\subset W_{i-1}(N) \subset W_{i}(N)\subset W_{i+1}(N)\subset \cdots \subset \V_\C.$$
  Denote by
$W^{j}_\bullet:=W_\bullet(\sum_{\alpha=1}^j N_\alpha) \mbox{ for } j=1,\cdots, l.$
We can choose a multivalued flat multigrading $$\V_\C=\sum_{\beta_1,\cdots,\beta_l}\V_{\beta_1,\cdots,\beta_l}$$ such that
$$\bigcap_{j=1}^lW_{\beta_{j}}^j=\sum_{k_j\leq \beta_j}\V_{k_1, \cdots,k_l}.$$

 Let $h$ be the Hodge metric on the PVHS $\sV.$
  In a special coordinate neighborhood $\triangle_1,$ let $v$ be a nonzero local multivalued flat section of $\V_{k_1, \cdots,k_l},$ then $(\widetilde{v})(z):=\exp(\frac{-\sum_{\alpha=1}^l\log z_\alpha N_\alpha }{2\pi\sqrt{-1}})v(z)$ is a local single-valued section of $\overline{\sV}.$ There holds a norm estimate (Theorem 5.21 in \cite{CKS86})
$$||\widetilde{v}(z)||_h\leq C^{'''} (\frac{-\log |z_1|}{-\log |z_2|})^{k_1/2}(\frac{-\log |z_2|}{-\log |z_3|})^{k_1/2}\cdots (-\log |z_l|)^{k_l/2}$$
on the region $$\Xi(N_1,\cdots, N_l):=\{(z_1,\cdots,z_l,\cdots, z_m)\in (\Delta^*)^l\times \Delta^{m-l}\,\,|\,\, |z_1|\leq |z_{2}|\leq \cdots \leq |z_l|\leq \varepsilon \}$$ for some small $\varepsilon>0,$ where $C^{'''}$ is a positive constant dependent on the ordering of $\{N_1, N_2, \cdots, N_l\}$ and $\varepsilon.$ Since the number of the ordering of $\{N_1,\cdots, N_l\}$ is finite,
for any  flat multivalued local section $v$ of $\V$ there exist
positive constants $C^{''}(\varepsilon)$ and $M^{''}$ such that
\begin{equation}\label{norm-estimate-1}
    ||\widetilde{v}(z)||_h\leq C^{''}(\varepsilon)(\prod_{\alpha=1}^l-\log |z_\alpha|)^{M^{''}} \,\,
\end{equation}
in the domain $\{(z_1,\cdots,z_l,\cdots, z_m)\,\,| 0<|z_i|<\varepsilon(i=1,\cdots,l), |z_j|<\varepsilon(j=l+1,\cdots,m) \}.$

Moreover,since the dual $\V^{\vee}$ is also a polarized real
variation of Hodge structures, we then have that for any  flat
multivalued local section $v$ of $\V$ there holds
\begin{equation}\label{norm-estimate-2}
    \frac{1}{C^{'}_1}(\prod_{\alpha=1}^l-\log |z_\alpha|)^{-M^{'}}\leq ||\widetilde{v}(z)||_h\leq C^{'}_1(\prod_{\alpha=1}^l-\log |z_\alpha|)^{M^{'}}
\end{equation}
in the domain $\{(z_1,\cdots,z_l,\cdots, z_m)\,\,| 0<|z_i|<\varepsilon(i=1,\cdots,l), |z_j|<\varepsilon(j=l+1,\cdots,m) \}$
for some suitable $\varepsilon>0,$ where
$C^{'}_1$ and $M^{'}$ are positive constants.
\begin{proposition}\label{Deligne-extension}
Let $X$ be an open K\"ahler manifold of dimension $m$ and  $\overline{X}$ a smooth compactification of $X$ such that the boundary $D:=\overline{X}-X$ is a simple normal crossing divisor. Let $(\triangle_1,z=(z_1,\cdots,z_l,\cdots, z_{m})) \subset \overline{X}$ be an arbitrary special coordinate neighborhood in which $D$ is given by $\prod_{i=1}^lz_i=0.$

Let $\V$ be a polarized real VHS on $X$ such that all local monodromies of $\V$ around the simple normal crossing boundary divisor are unipotent, and $h$ the Hodge metric on $\sV=\V\otimes \sO_X.$  Let $\sN$ be an arbitrary holomorphic subbundle of $\sV$ and $\overline{\sN}$ its Deligne's extension.

We have
$$\Gamma(\triangle_\varepsilon, \overline{\sN})=\{s\in \Gamma(\triangle_\varepsilon\cap X,\sN)\,\, |\,\, ||s||_h\leq C(\sum_{\alpha=1}^l-\log|z_\alpha|)^{M}) \mbox{ for some constants } M,C \},$$
%where $\triangle_\varepsilon=\varepsilon\triangle_1$
where $\Delta_\varepsilon$ is a scaling neighborhood of $\overline{X}$ for a sufficient small $\varepsilon>0.$
\end{proposition}
\begin{proof}
Let $(v_{\cdot})$ be a local flat multivalued basis  of $\sV$ over
$\triangle_1\cap X,$ and so we have a local basis
$(\widetilde{v_{\cdot}})(z)=\exp(\frac{-1}{2\pi\sqrt{-1}}\sum_{\alpha=1}^l\log z_\alpha N_\alpha )(v_{\cdot})(z)$
of $\overline{\sV}.$ Denote by $h_{ij}=<\widetilde{v_i}, \overline{\widetilde{v_j}}>_H.$
According to the estimate \ref{norm-estimate-2}, there are positive constants $C,M$ such that
\begin{equation}\label{norm-estimate-3}
    |h_{ij}|,\,\det (h_{ij}),\, (\det (h_{ij}))^{-1} \leq C(\sum_{\alpha=1}^l \log |z_\alpha|)^{2M}
\end{equation}
in a suitable neighborhood $\triangle_\varepsilon.$
One can use  Proposition 1.3 in \cite{Mum77} to
finish the proof.\\
\end{proof}

\begin{proof}[Proof of the theorem \ref{degenerate-Bergamn}]
\begin{myenumi}
\item By the lemma \ref{subbundle-tangent} and the theorem \ref{hodge-metric-is-Bergman-metric},
we can realize the holomorphic tangent bundle $\sT_{\sA_{g,\Gamma}}$ as a holomorphic subbundle of an Hodge bundle given by some PVHS, and the induced Hodge metric $H$ on the Siegel variety
$\sA_{g,\Gamma}=\Gamma\backslash\frak{H}_g$ is same as the
canonical Bergman metric.
We then finish the first statement by  using the estimate \ref{norm-estimate-2}.

\item By the statement (1), we have $\int_{\sA_{g,\Gamma}}|\omega_{\mathrm{can}} \wedge
\xi|<\infty$ for any smooth $(2(\dim\sA_{g,\Gamma})-2)$-form
$\xi$ on $\widetilde{\sA}_{g,\Gamma},$ and so the form $\omega_{\mathrm{can}}$ on $\sA_{g,\Gamma}$
defines a current $[\omega_{\mathrm{can}}]$ on $\widetilde{\sA}_{g,\Gamma}$
as follows :
$$<[\omega_{\mathrm{can}}], \phi>=\int_{\widetilde{\sA}_{g,\Gamma}}[\omega_{\mathrm{can}}] \wedge \phi :=\int_{\sA_{g,\Gamma}}\omega_{\mathrm{can}} \wedge \phi,$$  where $\phi$ is a smooth  $(2\dim(\sA_{g,\Gamma}) -2)$ form on $\widetilde{\sA}_{g,\Gamma}.$

We begin to show that $d[\omega_{\mathrm{can}}]=[d\omega_{\mathrm{can}}]=0.$ Let $\xi$ be any smooth $(2\dim(\sA_{g,\Gamma})-3)$-form on
$\widetilde{\sA}_{g,\Gamma}.$ Let $T_\delta$ be a tube
neighborhood of
$D_{\infty.n}=\widetilde{\sA}_{g,\Gamma}-\sA_{g,\Gamma}$ with
radius $\delta$ and $M_\delta:=\widetilde{\sA}_{g,\Gamma}\setminus
T_\delta.$ Then $\partial T_\delta=-\partial M_\delta.$
By definition,
$
  <d[\omega_{\mathrm{can}}], \xi> := -\int_{\sA_{g,\Gamma}}\omega_{\mathrm{can}}\wedge d\xi.
$
On the other hand, we have that
\begin{eqnarray*}
% \nonumber to remove numbering (before each equation)
  0=<[d\omega_{\mathrm{can}}],\xi>  &:=& \int_{\sA_{g,\Gamma}} d\omega_{\mathrm{can}}\wedge \xi \\
   &=& - \int_{\sA_{g,\Gamma}}  \omega_{\mathrm{can}}\wedge d\xi + \int_{\sA_{g,\Gamma}}  d(\omega_{\mathrm{can}}\wedge \xi)\\
  &=&- \int_{\sA_{g,\Gamma}}  \omega_{\mathrm{can}}\wedge d\xi+\lim_{\delta\to 0}\int_{M_\delta} d(\omega_{\mathrm{can}}\wedge \xi)\\
 \mbox{ (by Stoke's theorem)}&=&-\int_{\sA_{g,\Gamma}}  \omega_{\mathrm{can}}\wedge d\xi+\lim_{\delta\to 0} \int_{\partial M_\delta} \omega_{\mathrm{can}}\wedge \xi\\
&=&-\int_{\sA_{g,\Gamma}}  \omega_{\mathrm{can}}\wedge d\xi-\lim_{\delta\to
0}\int_{\partial T_\delta} \omega_{\mathrm{can}}\wedge
\xi\\
&=&-\int_{\sA_{g,\Gamma}}  \omega_{\mathrm{can}}\wedge d\xi.
\end{eqnarray*}
Here we use that $\omega_H$  has Poincar\'e growth on $D_\infty$ to obtain
$\lim\limits_{\delta\to 0}\int_{\partial T_\delta} \omega_{\mathrm{can}}\wedge \xi=0.$

\item Since $[\omega_{\mathrm{can}}]$ is a positive closed current, it is
a cohomology class on $\widetilde{\sA}_{g,\Gamma}$ of type
$(1,1).$ To prove that  $[\omega_{\mathrm{can}}]$ represents the first Chern class
$c_1(\Omega^1_{\widetilde{\sA}_{g,\Gamma}}(\log D_{\infty})),$  we
only need to show the following equality $$<[\omega_{\mathrm{can}}],
\eta>=<c_1(\Omega^1_{\widetilde{\sA}_{g,\Gamma}}(\log
D_{\infty})),\eta>$$ for any closed smooth
$(2\dim(\sA_{g,\Gamma})-2)$-form  $\eta$  on
$\widetilde{\sA}_{g,\Gamma}.$

Let $\eta$  be an arbitrary closed smooth
$(2\dim(\sA_{g,\Gamma})-2)$-form  on $\widetilde{\sA}_{g,\Gamma}.$
Let $\widetilde{H}$ be an arbitrary Hermitian metric on the bundle
$\Omega^1_{\widetilde{\sA}_{g,\Gamma}}(\log D_{\infty}).$ We have
\begin{eqnarray*}
% \nonumber to remove numbering (before each equation)
  <c_1(\Omega^1_{\widetilde{\sA}_{g,\Gamma}}(\log D_{\infty})), \eta> &:=&\frac{-1}{2\pi\sqrt{-1}}\int_{\widetilde{\sA}_{g,\Gamma}}  \Tr_{\widetilde{H}}(\Theta(\Omega^1_{\widetilde{\sA}_{g,\Gamma}}(\log D_{\infty}),\widetilde{H}))\wedge \eta\\
    &=& \frac{-1}{2\pi\sqrt{-1}}\int_{\widetilde{\sA}_{g,\Gamma}} \partial\overline{\partial}\log(\det
\widetilde{H})\wedge\eta \\
   &=& \frac{-1}{2\pi\sqrt{-1}}\int_{\sA_{g,\Gamma}} \partial\overline{\partial}\log(\det
\widetilde{H})\wedge\eta
\end{eqnarray*}
where $\Theta(\Omega^1_{\widetilde{\sA}_{g,\Gamma}}(\log
D_{\infty}),\widetilde{H})$ is the Chern form of
$(\Omega^1_{\widetilde{\sA}_{g,\Gamma}}(\log D),\widetilde{H}),$
and
\begin{eqnarray*}
% \nonumber to remove numbering (before each equation)
  <[\omega_{\mathrm{can}}], \eta> &:=&\int_{\sA_{g,\Gamma}}\omega_{\mathrm{can}}\wedge \eta  \\
   &=&  \frac{-1}{2\pi\sqrt{-1}}\int_{\sA_{g,\Gamma}}\Tr_{H_{\mathrm{can}}}\Theta(\Omega^1_{\sA_{g,\Gamma}},H_{\mathrm{can}})\wedge \eta \\
   &=& \frac{-1}{2\pi\sqrt{-1}}\int_{\sA_{g,\Gamma}} \overline{\partial}\partial\log(\det
H_{\mathrm{can}})\wedge\eta.
\end{eqnarray*}
Thus, it is sufficient to show that
$$\lim_{\delta\to 0} \int_{M_\delta} \overline{\partial}\partial\log(\frac{\det H_{\mathrm{can}}}{\det \widetilde{H}})\wedge\eta
=0.$$
We note that $\zeta:=\partial\log\det
H_{\mathrm{can}}-\partial\log\det\widetilde{H}$ is a global $(1,0)$-form on
$\sA_{g,\Gamma},$ we then get
\begin{eqnarray*}
% \nonumber to remove numbering (before each equation)
 \int_{M_\delta} \overline{\partial}\partial\log(\frac{\det H_{\mathrm{can}}}{\det \widetilde{H}})\wedge \eta
  &=& \int_{M_\delta} d\zeta\wedge \eta \\
 &=&-\int_{\partial T_\delta} \zeta\wedge \eta.
\end{eqnarray*}
As an application of the theorem\ref{hodge-metric-is-Bergman-metric},
we obtain that the $(1,0)$-form $\zeta$ near
the boundary divisor $D_{\infty}$ is nearly bounded in sense of Koll\'ar(cf.\cite{Kol87})
by Proposition 5.22 in \cite{CKS86}.
Thus, we have
$$\lim_{\delta\to 0}\int_{\partial T_\delta} \zeta\wedge \eta=0.$$

\item Since the metric connection form of any Hodge metric and its curvature form are both nearly bounded around the boundary divisor $D_\infty$(cf. Proposition 5.7 \cite{Kol87}), we
have :
$$C_1(\Omega^1_{\widetilde{\sA}_{g,\Gamma}}(\log D_{\infty}))^{\dim \sA_{g,\Gamma}}=(\frac{-1}{2\pi\sqrt{-1}})^{\dim \sA_{g,\Gamma}}\int_{\sA_{g,\Gamma}}\Tr_H(\Theta(\Omega^1_{\sA_{g,\Gamma}},H))^{\dim \sA_{g,\Gamma}} >0.$$
Since $\omega_{\widetilde{\sA}_{g,\Gamma}}(D_{\infty})$ is a
numerically effective line bundle on $\widetilde{\sA}_{g,\Gamma},$
we obtain that
$\omega_{\widetilde{\sA}_{g,\Gamma}}(D_{\infty})$ is a  big line
bundle by Siu's numerical criterion in \cite{Siu93}.
\end{myenumi}
\end{proof}
\begin{myrem}
We must point out that the statement (4) of Theorem \ref{degenerate-Bergamn} is first proven in \cite{Mum77} by using some calculations depending on a smooth toroidal compactification $\widetilde{\sA}_{g,\Gamma},$ and it can also be proven by  Siu-Yau's result on compactification (cf.Lemma 6 in \cite{SY82})
or by Zuo's result on positivity (cf.Theorem 0.1 in
\cite{Zuo00}). 
We can make an improvement on the statement (1) of the statement : By the generalized Schwarz lemma(cf.\cite{Yau78-2} and \cite{Roy80}), the metric $\omega_{can}$ is dominated above by the Poincar\'e metric  near infinity boundary of  not only smooth toroidal compactifications but also of  a general compactification with normal crossings boundary divisor. \\
\end{myrem}

\begin{lemma}[Moeller-Viehweg-Zuo cf.\cite{MVZ07}]\label{Mumford-extension}
Let $\Gamma$ be a neat arithmetic subgroup of $\Sp(g,\Z).$
Let $\overline{\sA}_{g,\Gamma}^{\mathrm{tor}}$ be a smooth toroidal compactification of the Siegel variety
$\sA_{g,\Gamma}:=\Gamma\backslash\frak{H}_g$ such that $D_{\infty}:= \overline{\sA}_{g,\Gamma}^{\mathrm{tor}}\setminus\sA_{g,\Gamma}$ is simple normal crossing.
Let $(\frak{L},\theta,h)$ be a homogenous Hodge bundle induced by PVHS on $\sA_{g,\Gamma}$ and $\sN$
any homogenous subbundle of $\frak{L}.$

Deligne's canonical
extension of the
bundle $\sN$ to  $\overline{\sA}_{g,\Gamma}^{\mathrm{tor}}$ coincides with the Mumford good extension(cf. \cite{Mum77}) of $\sN$ to  $\overline{\sA}_{g,\Gamma}^{\mathrm{tor}}$ by the Hodge
metric $h.$
\end{lemma}
\begin{proof}
It is a direct consequence of the estimates \ref{norm-estimate-1}, \ref{norm-estimate-2} and the proposition \ref{Deligne-extension}.
\end{proof}

\begin{lemma}\label{tangent bundle-Siegel varietie}
Let $\Gamma$ be a neat arithmetic subgroup of $\Sp(g,\Z).$ Let $\overline{\sA}_{g,\Gamma}^{\mathrm{tor}}$ be a smooth toroidal compactification of the Siegel variety
$\sA_{g,\Gamma}:=\Gamma\backslash\frak{H}_g$ such that $D_{\infty}:= \overline{\sA}_{g,\Gamma}^{\mathrm{tor}}\setminus\sA_{g,\Gamma}$ is simple normal crossing.

We have the following identifications
\begin{equation*}
   \sT_{\overline{\sA}_{g,\Gamma}^{\mathrm{tor}}}(-\log D_{\infty})=\Sym^2(\overline{E^{0,1}}),
\end{equation*}
and
$$\omega_{\overline{\sA}_{g,\Gamma}^{\mathrm{tor}}}(D_{\infty})=\bigwedge^{\dim_\C\sA_{g,\Gamma}}\Omega^1_{\overline{\sA}_{g,\Gamma}^{\mathrm{tor}}}(\log
D_{\infty})= (\det\overline{E^{1,0}})^{g+1}.$$

Moreover, the line bundle $\omega_{\overline{\sA}_{g,\Gamma}^{\mathrm{tor}}}(D_{\infty})$ is semi-positive
on the compactification $\overline{\sA}_{g,\Gamma}^{\mathrm{tor}}.$
%Therefore, both $\sT_{\sA_{g,\Gamma}}$ and $\Omega^1_{\sA_{g,\Gamma}}$  are $\Sp(g,\R)$-homogenous vector bundles on $\sA_{g,\Gamma}.$
\end{lemma}
\begin{proof}
We know that there is an inclusion $\sT_{\sA_{g,\Gamma}}\>\subset>>(E^{0,1})^{\otimes 2}.$
Since the Higgs field has the property $\theta\wedge \theta =0,$
the holomorphic subbundle $\Sym^2(E^{0,1})$ of  $(E^{0,1})^{\otimes 2}$ must contain the bundle $\sT_{\sA_{g,\Gamma}}.$
According to
$\rank_\C\sT_{\sA_{g,\Gamma}}=\rank_\C \Sym^2(E^{0,1})=g(g+1)/2,$
we obtain
$\sT_{\sA_{g,\Gamma}}=\Sym^2(E^{0,1}).$

The holomorphic vector bundle $\Sym^2(\overline{E^{0,1}})$ on $\overline{\sA}_{g,\Gamma}^{\mathrm{tor}}$ is Deligne's extension of
$\Sym^2(E^{0,1}).$ Using the proposition \ref{Mumford-extension} in the next subsection,  $\Sym^2(\overline{E^{0,1}})$ is also the unique Mumford's good extension of $\Sym^2(E^{0,1})$ by the Hodge metric $H.$ Shown in
Proposition 3.4 \cite{Mum77}ㄛ $\sT_{\overline{\sA}_{g,\Gamma}^{\mathrm{tor}}}(-\log
D_{\infty})$ is the unique Mumford's good extension of
$\sT_{\sA_{g,\Gamma}}$ by the metric $H.$ Therefore,
$$\sT_{\overline{\sA}_{g,\Gamma}^{\mathrm{tor}}}(-\log
D_{\infty})\cong\Sym^2(\overline{E^{0,1}}).$$
%We then get the rest immediately  by applying this identification.
Thus, we obtain that $\omega_{\overline{\sA}_{g,\Gamma}^{\mathrm{tor}}}(D_{\infty})= (\det\overline{E^{1,0}})^{g+1},$ so that $\omega_{\overline{\sA}_{g,\Gamma}^{\mathrm{tor}}}(D_{\infty})$ is semi-positive by Kawamata's positivity package in \cite{Kawa81}.
\end{proof}
\begin{myrem}
We can also get the semi-positivity of $\omega_{\overline{\sA}_{g,\Gamma}^{\mathrm{tor}}}(D_{\infty})$  by an argument of Mumford : it is shown in \cite{Mum77} that the sheaf $\omega_{\overline{\sA}_{g,\Gamma}^{\mathrm{tor}}}(D_{\infty})$ is the pull back of an ample line on the Satake-Baily-Borel compactification $\sA_{g,\Gamma}^*:=\Gamma\backslash\frak{H}_g^*.$
\end{myrem}

\begin{theorem}\label{stable-bundle}
Let $\Gamma\subset\Sp(g,\Z)$ be a neat arithmetic subgroup.
Let $\overline{\sA}_{g,\Gamma}^{\mathrm{tor}}$ be a smooth toroidal compactification of the Siegel variety
$\sA_{g,\Gamma}:=\Gamma\backslash\frak{H}_g$ such that $D_{\infty}:= \overline{\sA}_{g,\Gamma}^{\mathrm{tor}}\setminus\sA_{g,\Gamma}$ is a simple normal crossing divisor.

The logarithmic tangent bundle
$\sT_{\overline{\sA}_{g,\Gamma}^{\mathrm{tor}}}(-\log D_{\infty})$ is a stable
vector bundle with respect to the polarization $K_{\overline{\sA}_{g,\Gamma}^{\mathrm{tor}}}+D_\infty.$
%the canonical Bergman  current $[\omega_{\mathrm{can}}].$
\end{theorem}
\begin{proof}
Let $\mathcal{E}$ be  Deligne's canonical extension of the Hodge bundle
$(E^{\otimes 2},H),$ and $\overline{E^{0,1}}$ be Deligne's canonical extension of $E^{0,1}.$
By the lemma \ref{subbundle-tangent} and the lemma \ref{tangent bundle-Siegel varietie},  the logarithmic tangent
bundle $\sT_{\overline{\sA}_{g,\Gamma}^{\mathrm{tor}}}(-\log D_{\infty})=\Sym^2(\overline{E^{0,1}}),$ and so $\sT_{\overline{\sA}_{g,\Gamma}^{\mathrm{tor}}}(-\log D_{\infty})$ is a holomorphic subbundle  of $\mathcal{E}:=\overline{E}^{\otimes 2}.$

Let $\mathcal{G}$ be an arbitrary subbundle of $\mathcal{E}$ and
 $\widetilde{H}$  an arbitrary Hermitian metric on $\mathcal{G}.$
Let $\mathcal{G}_0:=\mathcal{G}|_{\sA_{g,\Gamma}}.$ We know
$\mathcal{G}$ is just Deligne canonical extension of
$\mathcal{G}_0.$ The degree of $\mathcal{G}$ with respect to the
the polarization $K_{\overline{\sA}_{g,\Gamma}^{\mathrm{tor}}}+D_\infty is $
is
\begin{eqnarray*}
% \nonumber to remove numbering (before each equation)
  \deg \mathcal{G} &:=& <c_1(\mathcal{G}), \bigwedge^{^{\dim \sA_{g,\Gamma}-1}}c_1(K_{\overline{\sA}_{g,\Gamma}^{\mathrm{tor}}}+D_\infty)> \\
   &=& <c_1(\mathcal{G}), [\omega^{\dim \sA_{g,\Gamma}-1}]>
\end{eqnarray*}
by (3) of the theorem \ref{degenerate-Bergamn} and Koll\'ar's argument of 5.18 in \cite{Kol87}.
Let $\eta:=\omega^{\dim \sA_{g,\Gamma}-1}.$  Similar calculation
as (3) of the theorem \ref{degenerate-Bergamn}, we have that
\begin{eqnarray*}
% \nonumber to remove numbering (before each equation)
 \deg\mathcal{G}&=&\int_{\overline{\sA}_{g,\Gamma}^{\mathrm{tor}}}  \Tr_{\widetilde{H}}(\Theta(\mathcal{G},\widetilde{H}))\wedge [\eta]\\
    &=& \frac{-1}{2\pi\sqrt{-1}}\int_{\overline{\sA}_{g,\Gamma}^{\mathrm{tor}}} \partial\overline{\partial}\log(\det
\widetilde{H})\wedge[\eta] \\
   &=& \frac{-1}{2\pi\sqrt{-1}}\int_{\sA_{g,\Gamma}} \partial\overline{\partial}\log(\det
\widetilde{H})\wedge\eta \\
   &=&\frac{-1}{2\pi\sqrt{-1}}\int_{\sA_{g,\Gamma}} \partial\overline{\partial}\log(\det
H)\wedge\eta + \frac{1}{2\pi\sqrt{-1}}\int_{\sA_{g,\Gamma}}
 \overline{\partial}\partial\log(\frac{\det H}{\det \widetilde{H}})\wedge\eta\\
&=&\frac{-1}{2\pi\sqrt{-1}}\int_{\sA_{g,\Gamma}}
\partial\overline{\partial}\log(\det H)\wedge\eta\\
&=&\int_{\sA_{g,\Gamma}}
\Tr_H(\Theta(\mathcal{G}_0,H))\wedge \omega_{\mathrm{can}}^{\dim \sA_{g,\Gamma}-1}\\
\end{eqnarray*}

Since the canonical Bergman metric is K\"ahler-Einstein, this essential
property implies that the logarithmic tangent bundle
$\sT_{\overline{\sA}_{g,\Gamma}^{\mathrm{tor}}}(-\log D_{\infty})$ is a
poly-stable vector bundle with respect to $[\omega].$

On the other hand, $\sA_{g,\Gamma}$ is simple, then we obtain that
the logarithmic tangent bundle $\sT_{\overline{\sA}_{g,\Gamma}^{\mathrm{tor}}}(-\log D_{\infty})$ can not be
decomposed into a direct sum by the argument
in the third paragraph of Page 272 in \cite{Yau87} and the argument of
Page 478-478 in \cite{Yau93}.\\
\end{proof}

%Section 2

%\newpage
\section{Some applications on  Siegel varieties}

\vspace{1cm}

All definitions and notations related to toroidal compactifications of Siegel varieties can be found in \cite{AMRT}, \cite{Chai} ,\cite{FC} and \cite{Y-Z}. We do not recite these definitions and notations in this section again, and use them freely.

Let $\frak{F}_0$ be the standard minimal cusp of the Siegel space $\frak{H}_g.$
Let $\Sigma_{\frak{F}_0}:=\{\sigma_\alpha^{\frak{F}_0}\}$ be a
suitable $\mathrm{GL}(g,\Z)$-admissible polyhedral decomposition of $C(\frak{F}_0)$ regular with respect to $\Sp(g,\Z)$
such that the induced symmetric $\Sp(g,\Z)$-admissible family
$\{\Sigma_\frak{F}\}_{\frak{F}}$ of polyhedral decompositions is
projective.

For any positive integer $l,$ let $\overline{\sA}_{g,l}$ to be
the symmetric toroidal compactification of the Siegel variety
$\sA_{g,l}:=\Gamma_g(l)\backslash\frak{H}_g$ constructed by
$\{\Sigma_\frak{F}\}_{\frak{F}},$ and let
$$D_{\infty,l}:=\overline{\sA}_{g,l}-\sA_{g,l}$$
the boundary divisor. For convenience, we write $\sA_{g}$ for $\sA_{g,1}.$\\
%and write $\overline{\lambda}_{n}$ for $\overline{\lambda}_{n,1}.$

%\subsection{Structures of morphisms $\overline{\lambda}_{n,m}:\overline{\sA}_{g,n}\to \overline{\sA}_{g,m}$ for all $m|n$ }
For  any positive integer $l,$ we sketch a key-step in the construction of the symmetric
compactification $\overline{\sA}_{g,l}$  as follows :

Let $\frak{F}$ be an arbitrary cusp of depth $k.$
$L_{\frak{F}}(l):=\Gamma(l)\cap U^{\frak{F}}(\Q)$ is a full
lattice in the vector space $U^{\frak{F}}(\C),$ and its dual is
$M_\frak{F}(l):=\Hom_\Z(L_{\frak{F}}(l),\Z).$ Explicitly, let
$\{\zeta_\alpha\}_{1}^{k(k+1)/2}$ be a lattice basis of
$L_{\frak{F}}:=\Sp(g,\Z)\cap U^\frak{F}(\Q)$ and
$\{\delta_\alpha\}_{1}^{k(k+1)/2}$ the associated dual basis of
$M_\frak{F}:= \Hom_\Z(L_{\frak{F}},\Z);$ then
$\{\zeta^l_\alpha:=l\zeta_\alpha \}_{1}^{k(k+1)/2}$ is a lattice
basis of $L_{\frak{F}}(l),$ and $\{\delta^l_\alpha:=
\frac{\delta_\alpha}{l}\}_{1}^{k(k+1)/2}$ is the dual basis of
$M_\frak{F}(l).$ For any cone $\sigma\in \Sigma_{\frak{F}},$ we
get a toroidal variety $X_{\sigma}(l):=\Spec\C[\sigma^\vee\cap
M_{\frak{F}}(l)];$  we then have
$$\mbox{$\widetilde{\Delta}_{\frak{F},\sigma}(l)$:= the interior of the
closure of $\frac{\frak{H}_g}{\Gamma(l)\cap U^{\frak{F}}(\Q)}$ in
$ X_{\sigma}(l)\times_{T_{\frak{F}}(l)}
\frac{D(\frak{F})}{\Gamma(l)\cap U^{\frak{F}}(\Q)}$ } $$ where
$T_{\frak{F}}(l):= \Spec\C[M_{\frak{F}}(l)]$ is a torus; gluing all
$\widetilde{\Delta}_{\frak{F},\sigma}(l)$ as $\sigma$ runs through
$\Sigma_{\frak{F}},$ we obtain an analytic variety
$Z_{\frak{F}}^{'}(l)$ and an open  morphism $\pi_{\frak{F}}^{'}(l):Z_{\frak{F}}^{'}(l)\to \overline{\sA}_{g,l}.$ %and obtain
As in Section 2 of \cite{Y-Z}, we define
$$Z_{\frak{F}}(l):=\frac{Z_{\frak{F}}^{'}(l)}{\Gamma(l)\cap
\sN(\frak{F})/\Gamma(l)\cap U^{\frak{F}}(\R)}.$$

%Since the
%decomposition $\Sigma_{\frak{F}_0}$ has non self-intersection, we
%can regard
%$\widetilde{\Delta}_{\frak{F}_0,\sigma^{\frak{F}_0}}(l)$ as an
%open neighborhood of $U_{[\frak{F}_0]}$ for any cone
%$\sigma^{\frak{F}_0}\in \Sigma_{\frak{F}_0}.$

Let $n,m$ be two positive integers with $m|n.$
%Suppose $\Sigma_{\frak{F}_0}$ has non $\Gamma(m)$-self-intersections.
%Then, by the lemma \ref{lemma for non-selfintersection-2}, $\Sigma_{\frak{F}_0}$ also has non $\Gamma_g(n)$-self-intersections.
We are going to
construct  a  natural morphism
$\overline{\lambda}_{n,m}:\overline{\sA}_{g,n}\to
\overline{\sA}_{g,m}.$ Given a cusp $\frak{F}$ and a cone
$\sigma\in \Sigma_{\frak{F}},$ the inclusion of the algebras
$\C[\sigma^{\vee}\cap
M_{\frak{F}}(m)]\>\subset>>\C[\sigma^{\vee}\cap M_{\frak{F}}(n)]$
induces a finite surjective morphism $\lambda^{\sigma}:
X_{\sigma}(n)\>>>X_{\sigma}(m).$ Therefore, we have an analytic
surjective morphism
$$\lambda_{\frak{F}}^{\sigma}: \widetilde{\Delta}_{\frak{F},\sigma}(n)\>>>\widetilde{\Delta}_{\frak{F},\sigma}(m),$$
such that any $\tau\prec\sigma$ there holds a commutative diagram
$$
\begin{CDS}
\widetilde{\Delta}_{\frak{F},\tau}(n)  \> \subset >\mbox{open
embedding}>
\widetilde{\Delta}_{\frak{F},\sigma}(n)\\
\V \lambda_{\frak{F}}^{\tau} V V \novarr  \V V \lambda_{\frak{F}}^{\sigma} V  \\
\widetilde{\Delta}_{\frak{F},\tau}(m)  \> \subset >\mbox{open
embedding}> \widetilde{\Delta}_{\frak{F},\sigma}(m)
\end{CDS},
$$
and so we obtain a morphism $\lambda_{\frak{F}}^{'}:
Z_{\frak{F}}^{'}(n)\to Z_{\frak{F}}^{'}(m)$ by gluing all
$\lambda_{\frak{F}}^{\sigma}\, \forall \sigma\in
\Sigma_{\frak{F}}.$ Since $\Gamma_g(n)$ is a normal subgroup of
$\Gamma_g(m),$  the morphism $\lambda_{\frak{F}}^{'}$ reduces to
the morphism
$$\lambda_{\frak{F}}: Z_{\frak{F}}(n)\to Z_{\frak{F}}(m).$$
It can be verified straightforwardly that $\lambda_{\frak{F}}$'s
are compatible with the morphisms $\Pi_{\frak{F}_1,\frak{F}_2}$'s
and the action of $\Gamma.$ Therefore, we have a global morphism
$$\overline{\lambda}_{n,m}:\overline{\sA}_{g,n}\to
\overline{\sA}_{g,m}.$$

Let $\sigma$ be an arbitrary topo-dimensional cone in
$\Sigma_{\frak{F}_0}.$ Consider the inclusion
$$0\>>>\C[\sigma^{\vee}\cap M_{\frak{F}_0}(m)]\>\subset>> \Q(\C[\sigma^{\vee}\cap M_{\frak{F}_0}(n)])$$
 where $\Q(\C[\sigma^{\vee}\cap M_{\frak{F}_0}(n)])$ is the quotient
field of the integral domain $\C[\sigma^{\vee}\cap
M_{\frak{F}_0}(n)].$ The algebra $\C[\sigma^{\vee}\cap
M_{\frak{F}_0}(n)]$ is indeed the integral closure of
$\C[\sigma^{\vee}\cap M_{\frak{F}_0}(m)]$ in
$\Q(\C[\sigma^{\vee}\cap M_{\frak{F}_0}(n)]).$ Then, the compactification
$\overline{\sA}_{g,n}$ is a normalization of the
morphism $\sA_{g,n}\to \overline{\sA}_{g,m}$ and so the morphism $\sA_{g,n}\to
\overline{\sA}_{g,m}$ factors through the morphism $\overline{\lambda}_{n,m}:
\overline{\sA}_{g,n}\to\overline{\sA}_{g,m}$(cf.\cite{FC}). Thus, we obtain the following  commutative diagram of morphisms
$$
\begin{CDS}
 \overline{\sA}_{g,n} \> \overline{\lambda}_{n,m} >> \overline{\sA}_{g,m} \\
 \novarr\SE  \overline{\lambda}_{n,1} E E \V V \overline{\lambda}_{m,1} V \\
\novarr \novarr \overline{\sA}_{g,1}.
\end{CDS}
$$

\begin{lemma}\label{boundary-divisor-relation-of-morphism}
Let $n,m\geq 1$ be two positive integers with $m|n.$
Let $\Sigma_{\frak{F}_0}:=\{\sigma_\alpha^{\frak{F}_0}\}$ be a
$\overline{\Gamma_{\frak{F}_0}}$(or $\mathrm{GL}(g,\Z)$)-admissible polyhedral decomposition of $C(\frak{F}_0)$
regular with respect to $\Sp(g,\Z).$ %where $\frak{F}_0$ is the standard minimal cusp of the Siegel variety $\frak{H}_g.$

%Assume that $\Sigma_{\frak{F}_0}$ has non $\Gamma(m)$-intersections.
Let $\overline{\sA}_{g,n}$(resp. $\overline{\sA}_{g,m}$) be the
symmetric toroidal compactification of $\sA_{g,n}$(resp.
$\sA_{g,m}$) constructed by  $\Sigma_{\frak{F}_0}.$
The morphism
$\overline{\lambda}_{n,m}:
\overline{\sA}_{g,n}\to\overline{\sA}_{g,m}$ has the following
property:
$$\overline{\lambda}_{n,m}^*D_{\infty,m}=\frac{n}{m}D_{\infty,n},$$
where $D_{\infty,m}:=\overline{\sA}_{g,m}\setminus\sA_{g,m}$ and
$D_{\infty,n}:=\overline{\sA}_{g,n}\setminus\sA_{g,n}.$
\end{lemma}
\begin{proof}
By the construction of boundary divisors of Siegel varieties from edges of the fan $\Sigma_{\frak{F}_0}$
in Theorem 2.22 of \cite{Y-Z}, %\ref{Infity-divisor-on-toroidal-compactification},
to study the relation between $D_{\infty,m}$ and $D_{\infty,n}$
is sufficient to study the morphism
$\lambda_{\frak{F}_0}^{\sigma_{\max}}:
\widetilde{\Delta}_{\frak{F}_0,\sigma_{\max}}(n)\to\widetilde{\Delta}_{\frak{F}_0,\sigma_{\max}}(m)$
for any top-dimensional cone $\sigma_{\max}$ in $\Sigma_{\frak{F}_0}.$

We can choose a basis
$\{\zeta_\alpha\}_{1}^{g(g+1)/2}$ of
$L_{\frak{F}_0}:=\Sp(g,\Z)\cap U^{\frak{F}_0}(\Z)$ such that
$$\sigma_{\max}=\{\sum_{\alpha=1}^{g(g+1)/2} \lambda_\alpha
\zeta_\alpha\,\, |\,\, \lambda_\alpha\in \R_{\geq 0},\, \,\,
\alpha=1,\cdots, g(g+1)/2\}.$$ Let $\{\delta_\alpha\}_{1}^{g(g+1)/2}$ be the dual basis of
$\{\zeta_\alpha\}_{1}^{g(g+1)/2}.$ Then
$$\sigma_{\max}^{\vee}=\{\sum_{\alpha=1}^{g(g+1)/2} \lambda_\alpha
\delta_\alpha\,\, |\,\, \lambda_\alpha\in \R_{\geq 0},\, \,\,
\alpha=1,\cdots, g(g+1)/2\}.$$
Since the inclusion $0\>>>\C[\sigma^{\vee}\cap
M_{\frak{F}_0}(m)]\>\subset>> \C[\sigma^{\vee}\cap
M_{\frak{F}_0}(n)]$ is of the following type
$$0\>>>\C[x_1,\cdots x_i,\cdots x_{g(g+1)/2}]\>\subset >> \C[\sqrt[\frac{n}{m}]{x_1},\cdots \sqrt[\frac{n}{m}]{x_i},\cdots \sqrt[\frac{n}{m}]{x_{g(g+1)/2}}],$$
we must have
$\overline{\lambda}_{n,m}^*D_{\infty,m}=\frac{n}{m}D_{\infty,n}.$\\
\end{proof}

\subsection{Spaces of Siegel cusp forms}

The Siegel space $\frak{H}_g$ has a global holomorphic coordinate system $\tau.$
Define a standard Euclidean form $d\mathcal{V}$ on $\frak{H}_g$ to be
$d\mathcal{V}_{\tau}:= \bigwedge_{1\leq i\leq j\leq g}d \tau_{ij}$ for $\tau=(\tau_{ij})_{1\leq i,j\leq g}\in\frak{H}_g.$
There is
$$d\mathcal{V}_{M(\tau)}= \det (C\tau+D)^{-(g+1)} d\mathcal{V}_{\tau}\, \, \mbox{ for } M=\left(
                                                                  \begin{array}{cc}
                                                                    A & B \\
                                                                    C & D \\
                                                                  \end{array}
                                                                \right)\in \Sp(g,\R).
$$

Let $\omega_{\frak{H}_g}$ be the canonical line bundle on $\frak{H}_g.$ For  any form $\varphi=f_{\varphi}\bigwedge\limits_{1\leq i\leq j\leq g}d
\tau_{ij}$ in $\Gamma(\frak{H}_g, \omega^{\otimes
k}_{\frak{H}_g}),$ there is an associated smooth positive
$(\frac{g(g+1)}{2}, \frac{g(g+1)}{2})$-form
$$
(\varphi\wedge \overline{\varphi})^{1/k}:=|f_\varphi|^{2/k}
\bigwedge_{1\leq i\leq j\leq
g}\frac{\sqrt{-1}}{2\pi}d\tau_{ij}\wedge d\overline{\tau_{ij}}.
$$

\begin{lemma}\label{cusp form-correspondence}
Let $n\geq 3,k\geq1,g\geq 2$ be integers.
Let $f\in \mathrm{M}_{k(g+1)}(\Gamma_g(n))$ be a modular form.
With respect to the correspondence
\begin{eqnarray*}
 \mathrm{M}_{k(g+1)}(\Gamma_g(n)) & \>\cong>> &  \Gamma(\frak{H}_g, \omega_{\frak{H}_g}^{\otimes
k})^{\Gamma_g(n)}:=\{s \in \Gamma(\frak{H}_g, \omega_{\frak{H}_g}^{\otimes k})\,\,|\,\, s \mbox{ is $\Gamma_g(n)$-invariant } \}\\
 f(\tau) & \longmapsto & \varphi_f:= f(\tau)(\bigwedge\limits_{1\leq i\leq
j\leq g}d \tau_{ij})^{\otimes k},
\end{eqnarray*}
the following two conditions are equivalent:
\begin{myenumii}
    \item $f\in \mathrm{S}_{k(g+1)}(\Gamma_g(n));$
    \item the holomorphic form $\varphi_f$ vanishes on all cusps of $\frak{H}_g.$
\end{myenumii}
Moreover, if $n\geq 3$ then (a) or (b) is equivalent to   the following\\
(c) $\int_{\sA_{g,n}}(\varphi_f\wedge \overline{\varphi_f})^{1/k}<\infty.$

\end{lemma}
\begin{proof}
Let $W_{g}$ be the one dimension isotropic real subspace of
$V_\R$ generated by $e_g.$

\noindent"$(a)\Leftrightarrow (b)$": We only show the case of $n=1,$ the others are similar. Suppose $f$ is a cusp form.
Then, $f$ vanishes on the cusp $\frak{F}(W_g),$ and so $f$
vanishes on any cusp $\frak{F}$ with $\frak{F}(W_g)\prec\frak{F}.$
Since $\varphi_f$ is a $\Gamma_g$-invariant form, $\varphi$
vanishes on all proper cusps of $\frak{H}_g.$
The converse part is obvious.\\

Assume that $n\geq 3 .$ We begin to show that "$(a)\Leftrightarrow (c)$":

\begin{itemize}
    \item Suppose $f$ is a cusp form. Then, $$f(\left(
\begin{array}{cc}
  \tau^{'} & 0 \\
  0 & \sqrt{-1}y
\end{array}%
\right))=O(\exp(-\frac{\pi}{2} y))\mbox{ for } y>>0,$$ and so
$\int_{\sA_{g,n}}(\varphi_f\wedge
\overline{\varphi_f})^{1/k}<\infty.$

    \item Suppose $f$ is not a cusp. Then, there is a
    $\tau^{'}\in\frak{H}_{g-1}$ such that $\Phi_n(f)(\tau^{'})\neq 0.$
Thus, there is a  neighborhood $U_{\tau^{'}}$ of fundamental domain such that $\tau^{'}$ is in the closure of $U_{\tau^{'}} $ 
such that
$|f(Z)|\geq c>0 $  on  $U_{\tau^{'}}$
for some positive constant $c.$ Therefore,
$$\int_{\sA_{g,n}}(\varphi_f\wedge
\overline{\varphi_f})^{1/k}\geq
\int_{U_{\tau^{'}}}(\varphi_f\wedge
\overline{\varphi_f})^{1/k}=\infty.$$
\end{itemize}
\end{proof}

\begin{corollary}\label{Koecher Principal-A(g,n)} Let $n\geq 3,k\geq1,g\geq 2$ be integers.
Let $\overline{\sA}_{g,n}$  be an arbitrary  smooth toroidal compactification of $\sA_{g,n}$
with simple normal crossing boundary divisor $D_{\infty,n}:=\overline{\sA}_{g,n}\setminus\sA_{g,n}.$

Then, we have :
$$\Gamma(\overline{\sA}_{g,n}, \omega_{\overline{\sA}_{g,n}}(D_{\infty,n})^{\otimes k})
   \cong \Gamma(\sA_{g,n}, \omega_{\sA_{g,n}}^{\otimes k})
   \cong \mathrm{M}_{k(g+1)}(\Gamma_g(n)),
$$
 $$\Gamma(\overline{\sA}_{g,n}, \omega_{\overline{\sA}_{g,n}}(D_{\infty,n})^{\otimes k-1}\otimes\omega_{\overline{\sA}_{g,n}}) \cong \mathrm{S}_{k(g+1)}(\Gamma_g(n)).$$
where $\omega_{\overline{\sA}_{g,n}}$ is the canonical line bundle on $\overline{\sA}_{g,n}$ and $\omega_{\sA_{g,n}}$ is the canonical line bundle on $\sA_{g,n}.$
\end{corollary}
\begin{proof}
With \cite{AMRT}, Mumford shows in \cite{Mum77} that
the canonical line bundle $\omega_{\sA_{g,n}}$ extends to an ample line bundle $L_{g,n}$ on $\sA_{g,n}^*$
and that the canonical
morphism $\overline{\pi}_{g,n}:\overline{\sA}_{g,n}\to \sA_{g,n}^* $ is proper with
$\overline{\pi}_{g,n}^*(L_{g,n})=\omega_{\overline{\sA}_{g,n}}(D_{\infty,n}).$

\begin{itemize}

\item Then,  $\sO_{\sA_{g,n}^*}=(\overline{\pi}_{g,n })_*\sO_{\overline{\sA}_{g,n}}$
and so
$(\overline{\pi}_{g,n})_*\omega_{\overline{\sA}_{g,n}}(D_{\infty,n})^{\otimes k}=(\overline{\pi}_{g,n})_*(\overline{\pi}_{g,n})^*L^{\otimes k}
\cong L^{\otimes k}.$
Thus,
$\Gamma(\sA_{g,n}^*, L^{\otimes k}_{g,n})\cong \Gamma(\sA_{g,n}^*, (\overline{\pi}_{g,n})_*\omega_{\overline{\sA}_{g,n}}(D_{\infty,n})^{\otimes k})
\cong \Gamma(\overline{\sA}_{g,n},\omega_{\overline{\sA}_{g,n}}(D_{\infty,n})^{\otimes k} ).$
Let $j: \sA_{g,n}\>\hookrightarrow>>\sA_{g,n}^*$ be open embedding. Since $\sA_{g,n}$ is normal and $\mathrm{codim}(\sA_{g,n}^*\setminus\sA_{g,n})=g\geq 2,$
we then have $j_*\omega_{\sA_{g,n}}^{\otimes k}=L^{\otimes k}.$
Thus,
$\Gamma(\sA_{g,n}^*, L^{\otimes k})\cong \Gamma(\sA_{g,n},\omega_{\sA_{g,n}}^{\otimes k}).$ That  $\Gamma(\sA_{g,n},\omega_{\sA_{g,n}}^{\otimes k})\cong \mathrm{M}_{k(g+1)}(\Gamma_g(n))$ is obvious.
\item By the lemma \ref{cusp form-correspondence} and the lemma \ref{lemma on ple-poles}, we have
$$\mathrm{S}_{k(g+1)}(\Gamma_g(n))\cong \{s \in \Gamma(\sA_{g,n}, \omega_{\sA_{g,n}}^{\otimes k})\,\,|\,\,\int_{\sA_{g,n}} (s\wedge \overline{s})^{1/k} < \infty \}.$$
Shown in Theorem 2.1 of \cite{Sak77}, there holds
$$\{s \in \Gamma(\sA_{g,n}, \omega_{\sA_{g,n}}^{\otimes k})\,\,|\,\,\int_{\sA_{g,n}} (s\wedge \overline{s})^{1/k} < \infty \}\cong\Gamma(\overline{\sA}_{g,n}, \omega_{\overline{\sA}_{g,n}}(D_{\infty,n})^{\otimes k-1}\otimes\omega_{\overline{\sA}_{g,n}}). $$
\end{itemize}
\end{proof}
\begin{myrem}Consider the short sequence
    $$0\>>>\omega_{\overline{\sA}_{g,n}}(D_{\infty,n})^{\otimes k-1}\otimes\omega_{\overline{\sA}_{g,n}} \>>>\omega_{\overline{\sA}_{g,n}}(D_{\infty,n})^{\otimes k}\>>>\omega_{\overline{\sA}_{g,n}}(D_{\infty,n})^{\otimes k}|_{D_{\infty,n}}\>>>0,$$
we have that $s\in \Gamma(\overline{\sA}_{g,n}, \omega_{\overline{\sA}_{g,n}}(D_{\infty,n})^{\otimes k})\cong \mathrm{M}_{k(g+1)}(\Gamma_g(n))$  is a cusp form if and only if $s|_{D_{\infty,n}}=0.$
Certainly, if $\overline{\sA}_{g,n}$ is projective then this result can also obtained by regarding $\overline{\sA}_{g,n}$ as the normalization of the blowing-up of $\sA_{g,n}^*$ along the ideal sheaf $\sJ$ supported on the subscheme $\sA_{g,n}^*\setminus\sA_{g,n}$(cf. Chap IV \cite{AMRT}).\\
\end{myrem}
%\begin{myenumii}%Let $\sN:=\omega_{\overline{\sA}_{g,n}}(D_{\infty,n})$ and $\sM:=\sO_{\overline{\sA}_{g,n}}(-D_{g,n})$

%\item
%That $s\in \Gamma(\overline{\sA}_{g,n}, \omega_{\overline{\sA}_{g,n}}(D_{\infty,n})^{\otimes k})\cong \mathrm{M}_{k(g+1)}(\Gamma_g(n))$  is a cusp form if and only if $s|_{D_{\infty,n}}=0.$
%Certainly, this result can also obtained by regarding the projective smooth toroidal compactification $\overline{\sA}_{g,n}$ as the normalization of the blowing-up of $\sA_{g,n}^*$ along the ideal sheaf $\sJ$ supported on the subscheme $\sA_{g,n}^*\setminus\sA_{g,n}$(cf. Chap IV \cite{AMRT}).

%\item Since the line bundle $\omega_{\overline{\sA}_{g,n}}(D_{\infty,n})$ is big,
%there is a finite positive  number $N_0$ such that
%$$\frac{\dim_\C\mathrm{S}_{k(g+1)}(\Gamma_g(n))}{k^{g(g+1)/2}}=\frac{\dim_\C\Gamma(\overline{\sA}_{g,n}, \omega_{\overline{\sA}_{g,n}}(D_{\infty,n})^{\otimes k-1}\otimes\omega_{\overline{\sA}_{g,n}})}{k^{g(g+1)/2}}>0   $$
%for any integer $k\geq N_0.$ \\
%\end{myenumii}

Since the logarithmic canonical line bundle $\omega_{\overline{\sA}_{g,n}}(D_{\infty,n})$ of any smooth compactification  $\overline{\sA}_{g,n}$ of the Siegel variety
$\sA_{g,n}$  is big,
there is a finite positive  number $N_0$ such that
$$\frac{\dim_\C\mathrm{S}_{k(g+1)}(\Gamma_g(n))}{k^{g(g+1)/2}}=\frac{\dim_\C\Gamma(\overline{\sA}_{g,n}, \omega_{\overline{\sA}_{g,n}}(D_{\infty,n})^{\otimes k-1}\otimes\omega_{\overline{\sA}_{g,n}})}{k^{g(g+1)/2}}>0   $$
for any integer $k\geq N_0.$ Actually, for dimensions of spaces of Siegel cusp forms,we have the following asymptotic formula which is probably well known to experts :
\begin{theorem}\label{estimate-dimension}
 Let $n\geq 3,k\geq1,g\geq 2$ be integers.
\begin{eqnarray*}
% \nonumber to remove numbering (before each equation)
  \limsup_{k\to \infty} \frac{\dim_\C\mathrm{S}_{k(g+1)}(\Gamma_g(n))}{k^{g(g+1)/2}}
   &=&[\Gamma_g(1): \Gamma(n)]\prod_{i=1}^g\zeta(1-2i),
\end{eqnarray*}
where $\zeta(s)$ is the Riemann-Zeta function.
\end{theorem}
\begin{proof} Let $\overline{\sA}_{g,n}$  be an arbitrary   smooth toroidal compactification of $\sA_{g,n}$
with simple normal crossing boundary divisor $D_{\infty,n}:=\overline{\sA}_{g,n}\setminus\sA_{g,n}.$ Let $\omega_{\overline{\sA}_{g,n}}$ be the canonical line bundle on $\overline{\sA}_{g,n}$ and $\omega_{\sA_{g,n}}$ the canonical line bundle on $\sA_{g,n}.$

 Define $L:=\omega_{\overline{\sA}_{g,n}}(D_{\infty,n}).$ Siegel's lemma says that there exists $C>0$ such that $$\dim_\C H^0(D_{\infty,n}, L^{\otimes k}|_{D_{\infty,n}}) \leq C k^{g(g+1)/2-1}\,\,\,\,
 \forall k\in \Z_{\geq 0},$$ then by the corollary \ref{Koecher Principal-A(g,n)} we get :
 \begin{eqnarray*}
 % \nonumber to remove numbering (before each equation)
 \limsup_{k\to \infty} \frac{\dim_\C\mathrm{S}_{k(g+1)}(\Gamma_g(n))}{k^{g(g+1)/2}}   &=& \limsup_{k\to \infty}\frac{\dim_\C H^0(\overline{\sA}_{g,n}, L^{\otimes k-1}\otimes\omega_{\overline{\sA}_{g,n}})}{k^{g(g+1)/2}} \\
    &=& \limsup_{k\to \infty}\frac{H^0(\overline{\sA}_{g,n}, L^{\otimes k})}{k^{g(g+1)/2}}.
 \end{eqnarray*}

We recall Demailly's holomorphic Morse inequalities in \cite{Dem89}:
Let $\widetilde{H}$ be an arbitrary Hermitian metric on the bundle
$L$ and $R(L,\widetilde{H})$ the curvature form of the metric connection of the Hermitian line bundle $(L,\widetilde{H}).$
For any non-negative integer $q,$ let $X(q,\widetilde{H})$ be the set $x$ of $\overline{\sA}_{g,n}$ such that $\frac{\sqrt{-1}}{2\pi}R(L,\widetilde{H})_x$
is non degenerate with exact $q$   negative eigenvalues. Set $X(\leq q, \widetilde{H}):=\bigcup\limits_{i=0}^qX(i,\widetilde{H}).$ For any non-negative integer $q,$  we have
$$\sum_{j=0}^q\dim_\C H^j(\overline{\sA}_{g,n}, L^{\otimes k})\leq \frac{k^{\frac{g(g+1)}{2}}}{(\frac{g(g+1)}{2})!} \int_{X(\leq q,\widetilde{H})} (-1)^q\bigwedge^{\frac{g(g+1)}{2}}c_1(L,\widetilde{H})+o(k^{\frac{g(g+1)}{2}})\,\,\mbox{   as }\,\, k\to \infty$$ with
equality for $q=g(g+1)/2.$ In particular, for any non-negative integer $q,$ there is the  weak More inequalities
$$\dim_\C H^q(\overline{\sA}_{g,n}, L^{\otimes k})\leq \frac{k^{\frac{g(g+1)}{2}}}{(\frac{g(g+1)}{2})!}
\int_{X(q,\widetilde{H})} (-1)^q\bigwedge^{\frac{g(g+1)}{2}}c_1(L,\widetilde{H})+o(k^{\frac{g(g+1)}{2}})\,\,\mbox{   as }\,\, k\to \infty.$$

We now use the arguments in section 2.3.3 of \cite{MM}
to show that for any integer $q\geq1,$
\begin{equation}\label{estimate-dimension-1}
\dim_\C H^q(\overline{\sA}_{g,n}, L^{\otimes k})= o(k^{\frac{g(g+1)}{2}})\,\,\mbox{   as }\,\, k\to \infty.
\end{equation}

In the lemma \ref{tangent bundle-Siegel varietie}, we  obtain  that $L$ is a numerically effective(nef) line bundle on  $\overline{\sA}_{g,n}.$
Therefore, for every small $\epsilon >0$ there is a smooth metric $H_\epsilon $ on $L$ such that $c_1(L,\widetilde{H})\geq -\epsilon \theta,$ where $\theta$ is a given positive $(1,1)$-form on $\overline{\sA}_{g,n}.$
On $ \overline{\sA}_{g,n},$  for any positive integer $q,$ we have
\begin{eqnarray*}
% \nonumber to remove numbering (before each equation)
 0\leq\frac{(-1)^q}{(\frac{g(g+1)}{2})!}c_1(L,H_\epsilon)^{g(g+1)/2}\chi_\epsilon &\leq & \frac{1}{q!}(\epsilon \theta)^q
\wedge\frac{(-1)^q}{(\frac{g(g+1)}{2}-q)!} (c_1(L,H_\epsilon)+\epsilon\theta) ^{g(g+1)/2-q} \\
   &\leq & \frac{1}{q!}(\epsilon \theta)^q
\wedge\frac{(-1)^q}{(\frac{g(g+1)}{2}-q)!} (c_1(L,H_\epsilon)+\theta) ^{g(g+1)/2-q},
\end{eqnarray*}
where $\chi_\epsilon$ is the characteristic function of $X(q,H_\epsilon).$ By Demailly's weak More inequalities, we then obtain
$$\dim_\C H^q(\overline{\sA}_{g,n}, L^{\otimes k})\leq \frac{k^{\frac{g(g+1)}{2}}\epsilon^q}{(\frac{q!g(g+1)}{2}-q)!}
\int_{ \overline{\sA}_{g,n}} [\theta]^q([c_1(L)]+[\theta])^{\frac{g(g+1)}{2}-q}+o(k^{\frac{g(g+1)}{2}})\,\,\mbox{   as }\,\, k\to \infty,$$
and so we obtain \ref{estimate-dimension-1}.

Since $L$ is  big by (4) of the theorem\ref{degenerate-Bergamn}, the Morse inequalities for $q=g(g+1)/2$ shows that
$$\limsup_{k\to \infty}\frac{H^0(\overline{\sA}_{g,n}, L^{\otimes k})}{k^{g(g+1)/2}}= \frac{1}{(\frac{g(g+1)}{2})!}\int_{\overline{\sA}_{g,n}}\bigwedge^{\frac{g(g+1)}{2}}c_1(L,\widetilde{H}).$$
By (4) of the theorem \ref{degenerate-Bergamn}, we actually get
$$\int_{\overline{\sA}_{g,n}}\bigwedge^{\frac{g(g+1)}{2}}c_1(L,\widetilde{H})=\int_{\sA_{g,n}}\bigwedge^{\frac{g(g+1)}{2}}\omega_{\mathrm{can}},$$
where $\omega_{\mathrm{can}}$ is the K\"ahler form of the canonical Bergman metric $H_{\mathrm{can}}$ on $\sA_{g,n}.$

Therefore, we obtain :
\begin{eqnarray*}
% \nonumber to remove numbering (before each equation)
   \limsup_{k\to \infty}\frac{H^0(\overline{\sA}_{g,n}, L^{\otimes k})}{k^{g(g+1)/2}}&=& \mathrm{Vol}(\sA_{g,n}) \\
   &=&[\Gamma_g(1): \Gamma(n)]\mathrm{Vol}(\sA_{g,1})  \\
   &=&[\Gamma_g(1): \Gamma(n)]\prod_{i=1}^g\zeta(1-2i).
\end{eqnarray*}
The last equality for volume can be found in \cite{Har}.
\end{proof}
\begin{proof}[Another proof of Theorem \ref{estimate-dimension}]
We have $$H^q(\overline{\sA}_{g,n}, \omega_{\overline{\sA}_{g,n}}(D_{\infty,n})^{\otimes k-1}\otimes\omega_{\overline{\sA}_{g,n}} )=0\,\,\, \mbox{ for }q\geq 1,k\geq 2$$
by  the Kawamata-Viehweg vanishing theorem(cf.\cite{EV}), then we
use the Riemann-Roch-Hirzebruch theorem to obtain $$\lim\limits_{k\to \infty} \frac{\dim_\C\mathrm{S}_{k(g+1)}(\Gamma_g(n))}{k^{g(g+1)/2}}
  =\mathrm{Vol}(\sA_{g,n})=[\Gamma_g(1): \Gamma(n)]\prod_{i=1}^g\zeta(1-2i).$$\\
\end{proof}

\subsection{General type of Siegel varieties with suitable level structures}
We can also get the following result by the theorem \ref{degenerate-Bergamn}.
\begin{corollary}[Mumford  cf.\cite{Mum77}]\label{Mumford-log-general-type}
Let $g\geq 1, n\geq 3$ be two integers. The Siegel variety $\sA_{g,n}$ is of
logarithmic general type.
\end{corollary}

%In this subsection, we study that how to use the method of cyclic covering to get a  variety of general-type from a variety of logarithmic
%general type.

From  the covering lemma  \ref{kawamata-covering-trick} and the  theorem \ref{general result on general type} in  A1, we immediately have :
\begin{corollary}[Mumford-Tai's Theorem  cf. Chap. IV. \cite{AMRT} and \cite{Mum77}]\label{Mumford-Tai's Theorem}
Let $g\geq 1,l\geq 3$ be two integers. There is a positive integer $N(g,l)$ such that the Siegel variety $\sA_{g, kl}$ is
of general type for any integer $k>N(g,l).$
\end{corollary}

Now we describe a relation between the existence of nontrivial cusp forms and the
type of manifolds : The existence of a nontrivial Siegel cusp form implies the general type of Siegel varieties with certain level structure.
%If the Siegel variety $\sA_{g,n}$($n\geq 3$) is not of general type then there
%does not exist nontrivial Siegel cusp form in $\mathrm{M}_{k(g+1)}(\Gamma_g)$ for any positive integer $k$ less than $n.$
Actually, the spaces of Siegel cusp forms supply the following  effective version of Mumford-Tai's theorem  : %(cf. Chap. IV. \cite{AMRT} and \cite{Mum77}):
\begin{theorem}\label{Theorem on general-type-1}

Let $l$ be an arbitrary positive integer and let  $g\geq 2$ be an integer.
If
$$N(g,l):=\min \{k\in \Z_{>0}\,\, |\,\, \dim_\C \mathrm{S}_{k(g+1)}(\Gamma_g(l))>0\}$$
is a finite number  then the Siegel variety $\sA_{g,Nl}$ is of general type
for any integer $N\geq \max\{\frac{3}{l},N(g,l)\}.$
\end{theorem}
\begin{myrem}
The theorem  \ref{estimate-dimension} guarantees that $N(g,l)$ is a finite integer if $g\geq2$ and $l\geq3.$ There are many examples from number theory
showing  that  $N(g,1)$ is finite for some low degree $g.$ %(cf. Example \ref{low dgree Siegel varieties}).\\
\end{myrem}

\begin{example}\label{low dgree Siegel varieties}
By the following list of examples of level one cusp forms for low degree $g,$ the Siegel varieties $\sA_{g,n}$ below are of general type :

(i) $\sA_{2,n}$ for $g=2$ and $n\geq 10,$ \,\, (ii) $\sA_{3,n}$ for $g=3$ and $n\geq 9,$ \,\,(iii) $\sA_{4,n}$ for $g=4$ and $n\geq8.$

\begin{itemize}
  \item Case $g = 2$ :  Igusa shows in \cite{Igu64} that there is a cusp form $\chi_{10,2}$ of weight $10$ with development
$$\chi_{10}\left(
             \begin{array}{cc}
               \tau_1 & z \\
               z      & \tau_2 \\
             \end{array}
           \right)=(\exp(2\pi\sqrt{-1}\tau_1)\exp(2\pi\sqrt{-1}\tau_2)+\cdots)(\pi z)^2+\cdots
$$
which vanishes  along the "diagonal" $z = 0$ with multiplicity $2.$ So the zero
divisor of $\chi_{10,2}$ in $\sA_2$ is the divisor of abelian surfaces that are products of elliptic
curves with multiplicity $2.$ Thus, there is a cusp form $\vartheta_2:= \chi_{10,2}^3\in \mathrm{S}_{10(2+1)}(\Gamma_2).$

%There is the Torelli map $\sM_2 \to \sA_2$ that associates
%to a hyperelliptic complex curve of genus $2$ given by $y_2 = f(x)$ its Jacobian.
%Then the pull back of $\chi_{10}$ to $\sM_2$ is related to the discriminant of $f$.

  \item Case $g = 3$ :  Tsuyumine shows in \cite{Tsu86} that the ring of classical modular forms $\oplus M_{k}(\Gamma_3)$ is generated by $34$ elements, and there is a cusp form $\chi_{18, 3}$ of weight $18,$ namely the product of the $36$ even theta constants $\theta[\epsilon].$ The zero divisor of $\chi_{18, 3}$ on $\sA_3$ is the closure
of the hyperelliptic locus. Thus, there is a cusp form $\vartheta_3:= \chi_{18,3}^2\in \mathrm{S}_{9(3+1)}(\Gamma_3).$

%This expresses the fact that a genus $3$ Riemann surface with a vanishing theta characteristic is hyperelliptic.

  \item Case $g = 4$ : Igusa shows in \cite{Igu81} that up to isometry there is
only one isomorphism class of even unimodular positive definite quadratic
forms in 8 variables, namely $E_8.$ In $16$ variables there are exactly two such
classes, $E_8\oplus E_8$ and $E_{16}.$ To each of these quadratic forms in $16$ variables we
can associate a Siegel modular form on $\Gamma_4$ by means of a theta series: $\theta_{E_8\oplus E_8}$
and $\theta_{E_{16}}.$ The difference $\chi_{8,4}:=\theta_{E_8\oplus E_8}-\theta_{E_{16}}$ is a cusp form of weight $8.$  The zero divisor of $\chi_{8,4}$ on $\sA_{4}$ is the closure of the locus of Jacobians of Riemann surfaces of genus $4$ in $A_4$(cf.\cite{Igu81} and \cite{Poo96}). Thus, there is a cusp form
$\vartheta_4:= \chi_{8,4}^5\in \mathrm{S}_{8(4+1)}(\Gamma_4).$
\end{itemize}
\end{example}

However these examples for low genus Siegel varieties of general type are  not optimal. Actually,
one has general type for $g=2, n\geq 4;$ $g=3, n \geq 3;$ $g=4,5,6, n\geq 2;$  $g\geq 7,$ $n\geq 1.$ The case $\sA_{6,1}$ is still open,
all other cases of low genus Siegel varieties are known to be rational or unirational.
Except for the case $\sA_{6,1},$ Hulek has completed the problem of general type for low genus Siegel varieties(cf.Theorem 1.1 \cite{Hul00}).\\

%\vspace{1cm}

In Section 3 of \cite{Y-Z}, we show that there are some restricted conditions to get a  projective smooth toroidal compactification of a Siegel variety with normal crossing boundary divisor.
To  prove the theorem \ref{Theorem on general-type-1}, we need  a projective smooth compactification of a Siegel variety with normal crossing boundary divisor.
For further studies on non locally symmetric varieties, we prefer the following  consequence of Hironaka's  Main theorem II to directly refining cone decomposition in smooth toroidal compactifications.
\begin{theorem}[Hironaka cf.\cite{Hir63}]\label{haronaka-1}
Let $C$ be reduced divisor on a nonsingular variety $W$ over a field
$k$ of characteristic zero.
There exists a sequence of
monoidal transformations $$\{\pi_j: W_j=Q_{Z_{j-1}}(W_{j-1})
\rightarrow W_{j-1}\,\,|\,\, 1\leq j\leq l\}$$
( $ Q_Z(X)\to X$ means the monoidal transform of $X$
with the center $Z$) and reduced divisor
$D_j$ on $W_j$ for $1\leq j\leq l$ such that
\begin{myenumiii}
\item $W_0=W, D_0=C,$

\item $D_j=\pi_j^{-1}(D_{j-1})$(Here $\pi_j^{-1}(D_{j-1})$ is defined to be
$\pi_j^*(D_{j-1})_{\mathrm{red}}$),

\item $Z_j$ is a nonsingular closed subvariety contained in $D_j,$

\item $W^*:=W_l$ is a nonsingular variety and $D_\infty:=D_l$ is a simple normal
crossing divisor on $W^*.$
\end{myenumiii}
Moreover, the map $\pi=\pi_l\circ\cdots\circ\pi_1$ is a proper birational
morphism from $W^*$ to $W$ such that the restriction morphism
$\pi|_{W^*\setminus D_\infty}: W^*\setminus D_\infty \>\cong >> W\setminus C$
is an isomorphism.
\end{theorem}

\begin{proof}[Proof of Theorem \ref{Theorem on general-type-1}]
We fix a $\overline{\Gamma_{\frak{F}_0}}$(or $\mathrm{GL}(g,\Z)$)-admissible polyhedral decomposition $\Sigma_{\frak{F}_0}$ of $C(\frak{F}_0)$ regular with respect to $\Sp(g,\Z)$ such
that the induced symmetric $\Sp(g,\Z)$-admissible family $\{\Sigma_\frak{F}\}_{\frak{F}}$ of polyhedral decompositions is
projective.
%The subjectivity of the morphism  $\overline{\pi}: \overline{\sA}_{g}\to \sA_{g}^*$ shows the following
%the short exact sequence  $$0\>>>\sO_{\sA_{g}^*}\>>>\overline{\pi}_*\sO_{\overline{\sA}_{g}}.$$
%The properness of $\overline{\pi}$ implies that $(\overline{\pi})_*\sO_{\overline{\sA}_{g}}$ is a finitely generated $\sO_{\sA_{g}^*}$-module.
%Since  $\sA_{g}^*$ and $\overline{\sA}_{g}$ have  same function field, then we obtain
% $$\sO_{\sA_{g}^*}=\overline{\pi}_*\sO_{\overline{\sA}_{g}}$$
%by that $\sA_{g}^*$ is a normal variety.

%since $L_{g,1}$ is very ample on $\sA^*_g.$ \\
%Therefore, $k_0$ is
%well-defined,
%and there is a nonzero cusp form $f\in \mathrm{S}_{k_0(g+1)}(\Gamma_g).$\\

Let $N\geq \max\{\frac{3}{l},N(g,l)\}$ be an integer
and define $n:=Nl.$
Since we fixed the decomposition
$\Sigma_{\frak{F}_{0}},$ we have the following commutative diagram
:
$$
\begin{CDS}
\overline{\sA}_{g,n} \> \overline{\lambda}_{n,l} >> \overline{\sA}_{g,l} \\
\V  \overline{\pi}_{g,n} VV \novarr \V V \overline{\pi}_{g,l} V \\
\sA_{g,n}^* \> \lambda_{n,l}^* >> \sA_{g}^*
\end{CDS}
$$
The lemma \ref{boundary-divisor-relation-of-morphism} shows
that there is
$$\overline{\lambda}_{n,l}^*D_{\infty,l}=ND_{\infty,n}$$
where $D_{\infty,l}:=\overline{\sA}_{g,l}\setminus\sA_{g,l}$ and $D_{\infty,n}:= \overline{\sA}_{g,n}\setminus\sA_{g,n}.$

The components of the boundary divisor $D_{\infty,n}=\overline{\sA}_{g,n}\setminus\sA_{g,n}$ may have self-intersections. However,
Hironaka's  results on resolution of singularities show that there exists a smooth compactification
$\widetilde{\sA}_{g,n}$ of $\sA_{g,n}$ and a proper birational morphism
\begin{equation}\label{Hironaka-proper-desingularty}
  \nu_n : \widetilde{\sA}_{g,n} \>>> \overline{\sA}_{g,n}
\end{equation}
such that
\begin{itemize}
  \item $\widetilde{D}_{\infty,n}=\nu^*(D_{\infty,n})_{\mathrm{red}}$ is a simple normal crossing divisor, $\widetilde{\sA}_{g,n}-\widetilde{D}_{\infty,n}=\sA_{g,n},$ and the restricted morphism $\nu_n|_{\sA_{g,n}}$ is the identity morphism;
  \item furthermore, write $D_{\infty,n}=\sum_{i=1}^l D_{i}$ and let $\widetilde{D}_i$ be the strict transform of $D_i,$ then
  $$\nu^*_n(D_{\infty,n})=\sum_{i=1}^l \widetilde{D}_i + E,\,\,\mbox{ and }\widetilde{D}_{\infty,n}=\sum_{i=1}^l \widetilde{D}_i + E_{\mathrm{red}}$$
(Here $E_{\mathrm{red}}$ is the exceptional divisor of $\nu_n$).
\end{itemize}

Then we get :
  $$\nu^*_n(D_{\infty,n})=\sum_{i=1}^l \widetilde{D}_i + E=\sum_{i=1}^l \widetilde{D}_i + E_{\mathrm{red}}+ (E-E_{\mathrm{red}})=
  \widetilde{D}_{\infty,n}+(E-E_{\mathrm{red}}).$$

Let $k_0:= N(g,l)$ and let $f\in \mathrm{S}_{k_0(g+1)}(\Gamma_g(l))$ be  a  non trivial cusp form.
Let $\theta_f:=f(d\mathcal{V})^{\otimes k}\in \Gamma(\sA_{g,n},
\omega_{\sA_{g}}^{\otimes k})$ where $d\mathcal{V}$ is the standard Euclidean form  on the Siegel space $\frak{H}_g.$
By the lemma \ref{cusp form-correspondence}, we have
$\int_{\sA_{g,n}} (\theta_f\wedge \overline{\theta_f})^{1/k} < \infty.$
Theorem 2.1 in \cite{Sak77} says that
$\theta_f$ defines a  $k$-ple
$\frac{g(g+1)}{2}$-form on $\widetilde{\sA}_{g,n}$ with at most
$(k-1)$-ple poles along $\widetilde{D}_{\infty,n},$
i.e., $$\vartheta_g\in H^0(\widetilde{\sA}_{g,n},
\omega_{\widetilde{\sA}_{g,n}}^{\otimes k_0}\otimes \sO_{\widetilde{\sA}_{g,n}}((k_0-1)\widetilde{D}_{\infty,n})).$$

Let $D_g$ be the Zariski closure in $\widetilde{\sA}_{g,n}$ of the zero divisor of $f$ on $\sA_{g,n},$ and let $m_f$ be the
vanishing order of $f$ at the cusp $\frak{F}(W_g),$
where $W_{g}$ is the one dimensional isotropic real subspace of
$V_\R$ generated by the vector $e_g.$
Since $\overline{\sA}_{g,n}$ can be regarded as the normalization of the blowing-up of $\sA_{g,n}^*$ along the ideal sheaf $\sJ$ supported on the subscheme $\sA_{g,n}^*\setminus\sA_{g,n},$, we have
$$\mathrm{div}(\vartheta_g)= Nm_f\nu^*_n(D_{\infty,n}) + D_g=Nm_f\widetilde{D}_{\infty,n}+Nm_f(E-E_{\mathrm{red}})+D_g.$$
Thus, we get
\begin{eqnarray*}
% \nonumber to remove numbering (before each equation)
   \omega_{\widetilde{\sA}_{g,n}}^{\otimes k_0}\otimes \sO_{\widetilde{\sA}_{g,n}}((k_0-1)\widetilde{D}_{\infty,n})
   &=& \sO_{\overline{\sA}_{g,n}}(\mathrm{div}(\vartheta_g))\\
   &=& \sO_{\overline{\sA}_{g,n}}(Nm_f\widetilde{D}_{\infty,n}+Nm_f(E-E_{\mathrm{red}})+D_g),
\end{eqnarray*}
and
$$\omega_{\overline{\sA}_{g,n}}^{\otimes k_0}=\sO_{\overline{\sA}_{g,n}}((Nm_f-k_0+1) \widetilde{D}_{\infty,n} + Nm_f(E-E_{\mathrm{red}})+ D_g).$$

Since $D_{\infty,n},$  $(E-E_{\mathrm{red}})$ and $D_g$ are all effective divisors on
$\overline{\sA}_{g,n},$  we  have that
$$\omega_{\overline{\sA}_{g,n}}(\widetilde{D}_{\infty,n})^{\otimes k_0} \subset
\sO_{\overline{\sA}_{g,n}}(h(\widetilde{D}_{\infty,n} +(E-E_{\mathrm{red}})+ D_g))\subset
\omega_{\overline{\sA}_{g,n}}^{\otimes hk_0}\,\,\mbox{ for
}\forall h> Nm_f.$$ Therefore $\omega_{\widetilde{\sA}_{g,n}}$
becomes a big line bundle on $\widetilde{\sA}_{g,n}$ by the corollary \ref{Mumford-log-general-type}.
\end{proof}

There is extensive work on the relationship between the existence of special modular forms and the geometry of moduli spaces of abelian varieties. The principal of using the existence of low weight cusp forms to study the Kodaira dimension of  moduli spaces of polarized abelian varieties is first used in \cite{GS96} and \cite{Grit95}, our theorem \ref{Theorem on general-type-1} provides a different version of this principal.
The principal is also efficient for studying moduli spaces of $K3$
surfaces and  moduli spaces of irreducible symplectic manifolds(cf.\cite{Kon93}\& \cite{GHS11});
 Gritsenko,Hulek and Sankaran  have proven that the moduli spaces of polarized K3 surfaces are general type(cf.\cite{GHS07}).\\

\noindent{\bf Acknowledgements.} We would like to thank Professor Ching-Li Chai and Professor Kang Zuo for useful suggestions.
The second author is grateful to Mathematics Department Harvard University for hospitality during 2009/2010.\\

% Section 3

\section{Appendix}

\vspace{1cm}

\subsection{A1. On general type varieties}
We will show  that one can obtain a variety of general type from any variety of logarithmic general type by covering  method. This subsection is parallel to part of work of Mumford in Section 4 of \cite{Mum77}, but our result is a little generalization and our technique is difficult from \cite{Mum77}.

Let $X$ be a complex manifold which is a Zariski open set of a compact complex manifold $\overline{X}$ such that  such that the boundary $D:=\overline{X}-X$ is divisor with at most simple normal crossing.Let $L$ be a holomorphic line bundle on $\overline{X}.$
For any positive integer $m,$ let $\Phi_{mL}$ be a meromorphic  map
define by a basis of $H^0(\overline{X},mL).$  The $L$-dimension of
$\overline{X}$ is defined to be
$$\kappa(L,\overline{X}):=\left\{
               \begin{array}{ll}
                 \max_{m\in N(L,X)}\{\dim_\C(\Phi_{mL}(\overline{X}))\}, & \hbox{if $N(L,\overline{X})\neq \emptyset$;} \\
                 -\infty, & \hbox{if $N(L,\overline{X})=\emptyset,$}
               \end{array}
             \right.
$$
where $N(L,\overline{X}):=\{m>0\,\, |\,\, \dim_\C
H^0(\overline{X},mL)>0\}.$ We call
$\kappa(X):=\kappa(\sO_{\overline{X}}(K_{\overline{X}}),
\overline{X})$ \textbf{Kodaira dimension} of $X,$ and
$\overline{\kappa}(X):=\kappa(\sO_{\overline{X}}(K_{\overline{X}}+D),
\overline{X})$ \textbf{logarithmic Kodaira dimension} of $X.$ $X$
is said to be of \textbf{general type}(resp. \textbf{logarithmic
general type}) if $\kappa(X)$(resp. $\overline{\kappa}(X)$) equals
to $\dim X.$ All definitions above are
independent of the choice of smooth compactification of $X$(cf.\cite{Iitaka77}).

%In this subsection, we study that how to use the method of cyclic covering to get a  variety of general-type from a variety of logarithmic
%general type.

\begin{lemma}\label{lemma on ple-poles}
Let  $(\overline{X}, D)$ be a compact complex manifold with a simple normal crossing divisor $D.$
Let $m, l$ be two arbitrary positive integers, we then have an
isomorphism
\begin{equation}
H^0(\overline{X},\sO_{\overline{X}}(mK_{\overline{X}}+lD))\cong \left\{
  \begin{array}{c}
    \mbox{$m$-ple $n$-form on $\overline{X}$ with at most} \\
    \mbox{ $l$-ple poles along $D$} \\
  \end{array}
\right\}.
\end{equation}
\end{lemma}
\begin{proof}
We have a system of  coordinates charts $\{(U_\alpha,  (z^\alpha_1,\cdots, z^\alpha_n))\}_\alpha$ on $\overline{X}$ satisfying $\overline{X}=\bigcup_\alpha U_\alpha.$
Let $\sigma$ be a
holomorphic section of $\sO_{\overline{X}}(D)$ defining $D.$ We can write $\sigma=\{\sigma_{\alpha}\}_\alpha$
such that $(\sigma_{\alpha})=D\cap U_\alpha$ with the rule $\sigma_\alpha=\delta_{\alpha\beta}\sigma_\beta \,\,\forall \alpha,\beta,$ where every $\delta_{\alpha\beta}$ is a transition function of the line bundle $\sO_{\overline{X}}(D).$

Let $\varphi \in
H^0(\overline{X},\sO(mK_{\overline{X}}+lD)$ be a global holomorphic section. We write $\varphi=\{\varphi_\alpha\}_\alpha$ such that
$$\varphi_\alpha=k^m_{\alpha\beta}\delta_{\alpha\beta}^{l}\varphi_\beta \,\,\mbox{ on } U_\alpha\cap U_\beta,$$ where every $k_{\alpha\beta}=\det(\frac{\partial z^\beta_i}{\partial
z^\alpha_j})$ is a transition function of the canonical line bundle $O_{\overline{X}}(K_{\overline{X}}).$
Then, we obtain the corresponding
 $m$-ple $n$-form $\omega=\{\omega_\alpha\}_\alpha$ on $\overline{X}$ as follows:
$$\omega_\alpha:=\frac{\varphi_\alpha}{\sigma_{\alpha}^{l}}(dz^\alpha_1\wedge\cdots \wedge
dz^\alpha_n)^m \,\, \mbox{ in }  U_\alpha.$$

Conversely, let $\omega$ be a $m$-ple $n$-form on $\overline{X}$ with at most $l$-ple poles along $D.$
Since $U_\alpha\cap D=\{z^\alpha_{i_1}\cdots z^\alpha_{i_t}=0\},$
we have
$$\omega_\alpha:=\omega|_{U_\alpha}=\frac{f_\alpha(dz^\alpha_1\wedge\cdots \wedge
dz^\alpha_n)^m}{(z^\alpha_{i_1})^{s_1}\cdots(z^\alpha_{i_t})^{s_t}}\,\, \mbox{ on } U_\alpha,$$
where every $s_i$ in an positive integer with $s_i\leq l.$
Then, we can write $\omega=\{\omega_\alpha\}_\alpha$ where
$$\omega_\alpha:=\omega|_{U_\alpha}=\frac{\varphi_\alpha(dz^\alpha_1\wedge\cdots \wedge
dz^\alpha_n)^m}{\sigma_\alpha^l} \,\,\mbox{ on  } U_\alpha.$$
Since $\omega_\alpha=\omega_\beta \mbox{ on } U_\alpha\cap U_\beta,$ then
$\varphi=\{ \varphi_\alpha\}_\alpha$ defines a global section in $H^0(\overline{X},\sO(mK_{\overline{X}}+lD))$ with
$\mathrm{div}(\varphi)\sim mK_{\overline{X}}+lD. $
\end{proof}

\begin{lemma}[Kawamata cf.\cite{EV}\&\cite{Kawa81}]\label{kawamata-covering-trick}
 Let $X$ be a $n$-dimensional quasi-projective nonsingular variety
and let $D=\sum_{i=1}^r D_i$ be a simple normal crossing divisor on
$X$. Let $d_1,\cdots,d_r$ be positive integers.
There exists
a quasi-projective nonsingular variety $Z$ and  a finite surjective
morphism $\gamma: Z\to X$ such that
\begin{myenumiii}
\item $\gamma^*D_i=N_j(\gamma^*D_i)_{red}$ for $i=1,\cdots,r;$
\item $\gamma^*D$ is a simple normal crossing divisor.
\end{myenumiii}
\end{lemma}
%\begin{proof}
%It is a special case of Kawamata's covering trick(cf.\cite{Kawa81} ).
%\end{proof}

%We have a general result on the finite coverings of logarithmic general type varieties.

\begin{theorem}\label{general result on general type}
Let $X$ be a complex non-singular quasi-projective
variety of logarithmic general type.
There is a
nonsingular quasi-projective variety  $Y$ of general type with a finite surjective morphism $f:Y\to X.$
\end{theorem}
\begin{proof}Let  $n=\dim_\C X.$
Let $\overline{X}$ be a projective smooth compactification of $X$
with a simple normal crossing  boundary divisor $B.$
Since $\overline{\kappa}(Y):=\kappa(K_{\overline{X}}+B, \overline{X})=n,$ it is shown by Sakai in Proposition 2.2 of \cite{Sak77} that for some integer $N>0$ there are
meromorphic differentials
$$\eta_0,\cdots,\eta_n\in
H^0(\overline{X},\sO_{\overline{X}}(NK_{\overline{X}}+(N-1)B))$$ such that
$\{\eta_i/\eta_0,\cdots,\eta_n/\eta_0\}$ is a transcendence  base of
the function field $\C(X).$

Let $d$ be an integer more than $N.$ We use Kawamata's covering trick by setting $d=N_1=N_2=\cdots$ in the lemma \ref{kawamata-covering-trick} to
get a projective manifold $\overline{Y}_d$
and a finite surjective morphism $f: \overline{Y}_d \to \overline{X}$  such
that $f^*(B)=dD_d,$ where  $D_d$ is a simple normal crossing divisor on  $ \overline{Y}_d.$\\

We begin to show that  $\overline{Y}_d$ is of
general type.

 For any $p\in \overline{Y}_d,$ we choose a local system of regular
coordinates $(z_1,\cdots,z_n)$ in the polycylindrical neighborhood
$U_p:=\{|z_1|\leq r_p,\cdots,|z_n|\leq r_p\}$ of $p$ such that if
$p\in D_d$ then the equation $z_1\cdots z_s=0$ defines $D_d$ around $p.$ For
$q=\pi_d(p),$ we choose a local system of regular
coordinates $(w_1,\cdots,w_n)$ in the polycylindrical neighborhood
$W_q:=\{|w_1|\leq r_q,\cdots,|w_n|\leq r_q\}$ of $q$ such that if
$q\in B$ then the equation $w_1\cdots w_t=0$ defines $B$ around $q.$

By definition, we have $(\overline{f})^{-1}(B)\subset D_d$ and
$f^*w_i=\prod_{j}z^{n_{ij}}_j\epsilon_i$ with  $n_{ij}\geq 0,$
where $\epsilon_i$ is an unit around $p.$
Thus we have
$f^*\frac{d w_i}{w_i}=\sum_j n_{ij}\frac{d z_j}{z_j}+\frac{d\epsilon_i}{\epsilon_i}\in \Omega^1(\log D_d)$
around the point $p.$

Let $\omega\in \{\eta_0, \cdots, \eta_n\}$ be an element. Since  $$H^0(\overline{X},\sO_{\overline{X}}(NK_{\overline{X}}+(N-1)B))\cong \left\{
  \begin{array}{c}
    \mbox{$N$-ple $n$-form on $\overline{X}$ with   } \\
    \mbox{ at most $(N-1)$-ple poles along $B$} \\
  \end{array}
\right\},
$$
we can write
$$\omega=g(w)\frac{(dw_1\wedge \cdots d w_n)^N}{w_1^{s_1}\cdots w_t^{s_t}}\,\mbox{ on } \,W_q$$
where $g(w)$ is a holomorphic function on $W_q$ and $s_1,\cdots,s_t$ are integers in $[0,N-1].$
Around the point $q,$ we then have
$$\omega=h(w)(\prod_{i=1}^t w_i)(\frac{dw_1\wedge \cdots d w_n}{{w_1}\cdots w_t})^N\,\mbox{ on } \,W_q$$
where $h(w)$ is a holomorphic function on $W_q.$
Since $f^*(B)=dD_d,$ we get that
$$f^*(\prod_{i=1}^t w_i)= (\prod_{j=1}^s z_i)^d\cdot\varepsilon   \,\,\,\mbox{ around } p$$
where $\varepsilon$ is a unit around $p,$
and we get that
$$f^*(\omega)=k(z)(\prod_{j=1}^s z_i)^d(\frac{dz_1\wedge \cdots d z_n}{{z_1}\cdots z_s})^N \,\,\,\mbox{ around } p$$
where $k(z)$ is a holomorphic function around $p.$
Thus each $f^*(\eta_i)$ is
regular on $\overline{Y}_d.$ The lemma \ref{lemma on ple-poles} says that all $f^*(\eta_1),\cdots,f^*(\eta_n)$ are in $H^0(\overline{X},\sO_{\overline{X}}(NK_{\overline{X}})).$  Therefore, $\{f^*(\frac{\eta_1}{\eta_0}), \cdots, f^*(\frac{\eta_n}{\eta_0})\}$
is a transcendence base of the function field $\C(\overline{Y}_d)$ and $\overline{Y}_d$ is of general type.\\
\end{proof}

\subsection{A2. On Siegel modular forms}
 Some materials related to the Satake-Baily-Borel compactification of a Siegel variety are taken from \cite{BB66}.

\begin{itemize}
\item Denote congruent groups by
\begin{equation}\label{lattice-group}
\Gamma_g(1):=\Sp(g,\Z),\,\,\Gamma_g(n):=\{\gamma\in
\Sp(g,\Z)\,\,|\,\, \gamma\equiv I_{2g} \mod n\}\,\, \forall n\geq
2.
\end{equation}
Obviously, each $\Gamma_g(n)$ is a  normal subgroup of
$\Sp(g,\Z)$ with finite index. For convenience, we write
$\Gamma_g$ for $\Gamma_g(1).$

     \item A subgroup $\Gamma\subset \Sp(g,\Q)$ is said to be \textbf{arithmetic} if
$\rho(\Gamma)$ is commensurable with $\rho(\Sp(g,\Q))\cap \GL(n,\Z)$ for some embedding
$\rho: \Sp(g,\Q)\> \hookrightarrow>> \GL(n,\Q).$ By a result of Borel, a subgroup $\Gamma\subset \Sp(g,\Q)$ arithmetic if and only if
$\rho^{'}(\Gamma)$ is commensurable with $\rho^{'}(\Sp(g,\Q))\cap \GL(n^{'},\Z)$ for every embedding
$\rho^{'}: \Sp(g,\Q)\> \hookrightarrow>> \GL(n^{'}).$ %(cf.Chap. VI. \cite{Mil}).
We note that a subgroup $\Gamma\subset \Sp(g,\Z)$ is arithmetic if and only if
$[\Sp(g,\Z):\Gamma]<\infty.$
     \item Let $k^{'}$ be a subfield of $\C.$ An automorphism $\alpha$ of a $k^{'}$-vector space is defined to be \textbf{neat}
(or \textbf{torsion free}) if its eigenvalues in $\C$ generate a torsion free subgroup of $\C.$
An element $h\in \Sp(g,\Q)$  is said to be \textbf{neat}(or \textbf{torsion free}) if $\rho(h)$ is neat for one faithful representation $\rho$ of $\Sp(g,\Q).$  A subgroup $\Gamma\subset\Sp(g,\R)$ is \textbf{neat} if all elements of $\Gamma$ are torsion free.
 It is known that if $n\geq 3$ then $\Gamma_g(n)$  is a neat arithmetic subgroup of $\Sp(g,\Q).$
\end{itemize}

The Siegel space $\frak H_g$ of degree $g$ is a complex manifold defined to be
the set of all symmetric matrices over $\C$ of degree $g$ whose imaginary parts are positive defined.
%Thus, we have $$\frak{H}_g=\frak{Y}_g+\frak{Y}_g^*,$$ where $\frak{Y}_g$ is the vector space of $(g\times g)$-symmetric matrices and $\frak{Y}_g^*$ is the cone of positive-definite matrices in $\frak{Y}_g,$
%and so $\frak{H}_g$ is a tube domain.
The action of $\Sp(g,\R)$ on $\frak{H}_g$ is defined as
$$\left(
         \begin{array}{cc}
           A & B \\
           C & D \\
         \end{array}
       \right)\bullet\tau:=\frac{A\tau+B}{C\tau+D}.$$
It is known that the simple real Lie group $\Sp(g,\R)$ acts on
$\frak{H}_g $ transitively.
   A \textbf{Siegel variety} is defined
to be
$\sA_{g,\Gamma}:=\Gamma\backslash\frak{H}_g,$
where $\Gamma$ is an arithmetic subgroup of $\Sp(g,\Q).$
\begin{itemize}
  \item A Siegel variety $\sA_{g,\Gamma}$ is a normal quasi-project variety.
  \item Any  neat arithmetic subgroup $\Gamma$ of $\Sp(g,\Q)$ acts freely on the Siegel Space $\frak{H}_g,$ then the induced
$\sA_{g,\Gamma}$ is a regular quasi-projective complex variety of
dimension $g(g+1)/2.$

  \item A \textbf{Siegel variety of degree $g$ with level $n$} is defined to be
$$\sA_{g,n}:=\Gamma_g(n)\backslash\frak{H}_g. $$
Thus, the Siegel varieties $\sA_{g,n}\,\,$($n\geq 3$)  are
quasi-projective complex manifolds.\\
\end{itemize}

The Siegel space $\frak H_g$ can be realized as a bounded domain parameterizing weight one polarized Hodge structures :
\begin{myprop1}[Borel's embedding  cf.\cite{Sat}\&\cite{Del73}]\label{Borel-embedding}
Define
$$
\frak{S}_g=\frak{S}(V_\R,\psi):=\{F^1\in \mathrm{Grass}(g,
V_\C)\,\,)\,\, | \,\, \psi(F^1,F^1)=0, \sqrt{-1}\psi(F^1,
\overline{F^1})>0\}.
$$
The map $\iota: \frak{H}_g\>\cong>> \frak{S}_g\,\,\tau \mapsto F^1_\tau$
identifies the Siegel space $\frak{H}_g$ with the period domain
$\frak{S}_g,$ where
$F_{\tau}^1:=\mbox{the subspace of $V_\C$ spanned by the column vectors of  $\left(
                                                                           \begin{array}{c}
                                                                             \tau \\
                                                                             I_g
                                                                           \end{array}
                                                                         \right).$ }$
Moreover, the map $h$ is biholomorphic.
\end{myprop1}

We set
\begin{eqnarray*}
% \nonumber to remove numbering(before each equation)
\overline{\frak{S}_g}   &:=& \{F^1\in \mathrm{Grass}(g, V_\C)\,\,|\,\, \psi(F^1,F^1)=0, \sqrt{-1}\psi(F^1, \overline{F^1})\geq 0\}, \\
 \partial\frak{S}_g  &:=& \{F^1\in \check{\frak{S}}_g \,\,|\sqrt{-1}\psi(F^1, \overline{F^1})\geq 0, \psi(\cdot, \overline{\cdot}) \mbox{ is degenerate on } F^1\}.\\
                     &=&\{F^1\in \overline{\frak{S}_g} \,\,|F^1\cap \overline{F^1} \mbox{ is a non trivial isotropic space }\}
\end{eqnarray*}
A \textbf{boundary component} of the Siegel space $\frak{S}_g=\frak{S}(V,\psi)$ is  a subset in $\partial\frak{S}_g$ given by
$$\frak{F}(W_\R):=\{ F^1\in \overline{\frak{S}_g} \,\,|\,\,  F^1\cap \overline{F^1}=W_\R\otimes\C \mbox{ where $W_\R$ is an isotopic real subspace of $V_\R$ } \}.$$
A boundary component $\frak{F}(W)$ of the Siegel
space $\frak{S}_g$ is rational(i.e., $\frak{F}(W)$ is a cusp) if and only if $W=W_\Q\otimes\R,$
where $W_\Q$ is an isotropic  subspace  of $(V_\Q,\psi).$

Define
$\frak{H}_g^*:=\bigcup\limits_{\mbox{ cusp }\frak{F } }
\frak{F}.$
The set $\frak{H}_g^*$ is stable under the action of
$\Sp(g,\Q).$ Actually, the set $\frak{H}_g^*$ is a disjoint union of locally closed $\Sp(g,\Q)$-orbits
$\frak{H}_g^*=\sO_0\bigcup\limits^{\circ}\sO_1\bigcup\limits^{\circ}\cdots\bigcup\limits^{\circ}\sO_{g},$
and each orbit $\sO_r$ is a set of disjoint union rational boundary components with same rank, i.e.,
$$\sO_{g-r}:=\bigcup^{\circ}_{\mbox{$\frak{F}(W)$ with } \dim_\R W=r} \frak{F}(W).$$
In particular, $\sO_g=\frak{F}(\{0\})=\frak{H}_g.$
Let $\Gamma$ be an arbitrary arithmetic
subgroup of $\Sp(g,\Q).$  Let $\sA_{g,\Gamma}^*:=\Gamma\backslash\frak{H}_g^*$ be  the \textbf{Satake-Baily-Borel compactification} of $\sA_{g,\Gamma}.$ It is known that the analytic variety $\sA_{g,\Gamma}^*$  has an algebraic
structure as a normal projective complex variety.

Let $l$ be an arbitrary positive integer.
It is known that there is
$\frak{F}(W)\cong \frak{H}_{g-r}$
for any  cusp with $\dim_\R W=r.$
Therefore, the quotient
$\Gamma_g(l)\backslash\sO_{g-r}$ is a disjoint union of
$[\Sp(g,\Z):\Gamma_g(l)]$ locally closed subsets, and each
disjoint component of $\Gamma_g(l)\backslash\sO_r$ is canonically
isomorphic to the Siegel variety $\sA_{g-r, l}.$
%If $r>0,$ all these locally closed subschemes are called \textbf{cusp subschemes}.
The Satake-Baily-Borel compactification $\sA_{g,l}^*$ of
the Siegel variety $\sA_{g,l}$ is
$$\sA_{g,l}^*:=\Gamma_g(l)\backslash \frak{H}_g^*= (\Gamma_g(l)\backslash\sO_0)\bigcup^{\circ}(\Gamma_g(l)\backslash\sO_1)
\bigcup^{\circ}\cdots\bigcup^{\circ}(\Gamma_g(l)\backslash\sO_{g}).$$
In particular, we have
$\sA_{g}^*=\sA_{g}\bigcup\limits^{\circ}\sA_{g-1}\bigcup\cdots\bigcup\limits^{\circ}\sA_{0}.$
The construction of Satake-Baily-Borel compactifications shows
there is a  natural morphism
\begin{equation}\label{morphism-BBS}
    \lambda_{n,m}^*: \sA_{g,n}^* \>>> \sA_{g,m}^*
\end{equation}
for any two positive integers $m,n$ with $m|n.$

\begin{definition}[Cf.\cite{Fre}]
Let $k\geq 1,n\geq 1,g\geq2$ be integers. A complex-valued
function on $\frak{H}_g$ is called a \textbf{Siegel modular form
of weight $k,$ degree $g$ and level $n$} if the following
conditions are satisfied:
\begin{itemize}
    \item $f:\frak{H}_g\to \C$ is a holomorphic function;

    \item $ f(\tau)= (f|\gamma)(\tau):=\det (C\tau+D)^{k}f(\gamma(\tau))
\,\,\,\,\forall \gamma=\left(
                                                           \begin{array}{cc}
                                                             A & B \\
                                                             C & D \\
                                                           \end{array}
                                                         \right)
\in \Gamma_g(n).$
\end{itemize}
Denote by
$$\mathrm{M}_k(\Gamma_g(n)):=\{ \mbox{Siegel modular forms of weight $k,$ degree $g$ and level $n$}  \}.$$
\end{definition}

Recall some important properties of Siegel modular forms (Cf.\cite{Fre}):
 Let $f\in \mathrm{M}_{k}(\Gamma_g(n))$($g\geq 2$). The Siegel modular form $f$ has an
expansion of the form
\begin{equation}\label{fourier expansion}
    f(\tau)=\sum_{2A\in \mathrm{Sym}_g(\Z),A\geq 0}c(A)\exp(\frac{\sqrt{-1}\pi}{n}\mathrm{Tr}(A\tau))
\end{equation}
where $c(A)$ are constant coefficients. The series \ref{fourier
expansion} converges absolutely on $\frak{H}_g$ and uniformly on
each subset of $\frak{H}_g$ of the form $W^g_\epsilon=\{
X+\sqrt{-1}Y\in\frak{H}_g\,\,|\,\, Y\geq \epsilon I_{2g} \} \mbox{
with $\epsilon>0.$ }$ In particular, $f$ is bounded on each
subsets. All coefficients $c(A)$ for $2A\in \mathrm{Sym}_g(\Z)$ and $A\geq 0$ satisfy
\begin{equation}\label{fourier expansion coefficients}
    c(^tVAV)=(\det(V))^k\exp(-\frac{\sqrt{-1}\pi}{n}\mathrm{Tr}(AVU))c(A)
\end{equation}
for any $M\in \Gamma_g(n)$ of the form
$M=M(V,U)=\left(%
\begin{array}{cc}
  V^{-1} & U \\
  0      & ^tV \\
\end{array}%
\right).$ The series \ref{fourier expansion} is called the
Fourier expansion of $f,$ and any $c(A)$ with $2A\in
\mathrm{Sym}_g(\Z)$ and $ A\geq 0$ is a Fourier coefficient
of $f.$

Let $n$ be an arbitrary positive integer.
Define $\tau_{t}= \left(
\begin{array}{cc}
  \tau^{'} & 0 \\
  0 & \sqrt{-1}t
\end{array}%
\right) $ in $\frak{H}_{g}$ with $\tau^{'}\in\frak{H}_{g-1}, t>0.$ By Proposition 1.3 in Section 1 \cite{Y-Z}, $\lim\limits_{t\to
\infty}\tau_{t}$ corresponds to the following  element in
$\partial\overline{\frak{S}_g}$
$$F_{\tau^{'},\infty}:=\mbox{ the subspace of $V_\C$ spanned by the column vectors
of $\left(
\begin{array}{cc}
   \tau^{'} & 0\\
   0        & 1\\
   I_{g-1}  & 0\\
   0        & 0
\end{array}
\right).$ }$$ Let $W_{g}$ be the one dimension isotropic real subspace of
$V_\R$ generated by $e_g.$ $F_{\tau^{'},\infty}$ is actually in
the cusp $\frak{F}(W_{g}).$ The properties \ref{fourier expansion} and \ref{fourier expansion coefficients} guarantee
that  any modular $f\in \mathrm{M}_{k}(\Gamma_g(n))$ can extend to
be a holomorphic function on
$$\frak{H}_g\bigcup \bigcup_{\mbox{Cusp $\frak{F}$ with }
\frak{F}(W_{g})\preceq \frak{F} }\frak{F},$$ we then have
\begin{eqnarray*}
  f(F_{\tau^{'},\infty}) &=&  \lim_{t\to \infty}f(\left(
\begin{array}{cc}
  \tau^{'} & 0 \\
  0 & \sqrt{-1}t
\end{array}%
\right) ):=\Phi_n(f)(\tau^{'}) \\
   &=&\sum_{2A^{'}\in \mathrm{Sym}_{g-1}(\Z),A^{'}\geq 0}c(
\left(%
\begin{array}{cc}
  A^{'} & 0 \\
  0     & 0 \\
\end{array}%
\right))\exp(\sqrt{-1}\pi\mathrm{Tr}(A^{'}\tau^{'})).
\end{eqnarray*}
Therefore, for $g\geq 2$ we can define the \textbf{Siegel
operators}
$\Phi_n: \mathrm{M}_{k}(\Gamma_g(n))\to \mathrm{M}_{k}(\Gamma_{g-1}(n))$ by sending $f$ to $\Phi_n(f)$ for all positive integer $n.$
For any integers $n\geq 1,k\geq1,g\geq 2,$
$$\mathrm{S}_{k}(\Gamma_g(n)):=\{ f\in M_{k}(\Gamma_g(n))\,\,\,|\,\, \Phi_n(f|\gamma)=0\mbox{ for }\forall \gamma\in \Gamma_g\}.$$
is a set of all \textbf{Siegel cusp forms} of weight $k,$ degree $g$ and level $n.$\\

%\begin{proof}
%It can be obtained from Lemma \ref{tangent bundle-Siegel varietie} and Proposition \ref{Deligne-extension}.
%\end{proof}

%\vspace{1cm}

%\vspace{1cm}
%\newpage
%%%%%%%%%%%%%%%%%%%%%%% References %%%%%%%%%%%%%%%%%

%%%%%%%%%%%%%%%%%%%%%%%% End References %%%%%%%%%%%%%%%%%


\begin{thebibliography}{Bibliography} % {Transcendental Technique in Algebraic Geometry} % XXX=breitestes Label
%\bibitem[Gr1]{Gr}Griffths, P.: Topices in transcendental algebraic geometry.
%Ann of Math. Stud. 106 (1984) Princeton Univ. Press. Princeton, N.J.
\bibitem[AMRT]{AMRT} Ash, A.; Mumford, D.; Rapoport, M.; Tai, Y. :
{\it Smooth compactifications of locally symmetric varieties. Second edition. With the collaboration of Peter Scholze.}
Cambridge Mathematical Library. Cambridge University Press, Cambridge, 2010.


\bibitem[BB66]{BB66} Baily, W.L. \& Borel, A. :{\it Compactification of arithmetic quotients of bounded
symmetric domains.} Ann. of Math. 84(1966) 442?28.%, MR 0216035, Zbl 0154.08602.

\bibitem[Chai]{Chai} Chai, C.-L. :{\it Siegel moduli schemes and their compactifications over $\C$.} Arithmetic geometry(Storrs, Conn., 1984), 231--251, Springer, New York, 1986.

\bibitem[CK87]{CK87} Cattani,E. ;  Kaplan, A. :{\it Degenerating variations of Hodge structure.} Actes du Colloque de Th廩rie de Hodge (Luminy, 1987).  Ast廨isque  No. 179-180  (1989), 9, 67--96.
\bibitem[CKS86]{CKS86}  Cattani,E. ;  Kaplan, A. \&  Schmid, W. :{\it Degeneration of Hodge
structures.} Ann. Math., {\bf 38}, (1986), 457-535.

\bibitem[Del71]{Del71}  Deligne, P. :{\it Th\'eorie de Hodge II.} I.H.\'E.S. Publ. Math.
{\bf 40} (1971) 5--57
\bibitem[Del73]{Del73}Deligne, P. :{\it Espaces hermitiens sym\'etriques.} Notes from a seminar at IHES,
Spring 1973. English version :{\it Deligne' Notes on Symmetric
Hermitian Spaces.} Notes taken by Milne, J.S. at website
http://www.jmilne.org/math/
\bibitem[Del79]{Del79}Deligne, P. :{\it Vari\'et\'es de Shimura: interpretation modulaire, et techniques de construction de mod悶les canoniqes.} Proc. Symp. Pure Math., A.M.S. 33, part 2, 1979, pp. 247-290.

\bibitem[Dem89]{Dem89}Demailly, J.-P. :{\it Holomorphic Morse inequalities.}  Several complex variables and complex
geometry, Part 2 (Santa Cruz, CA, 1989), Proc. Sympos. Pure Math.,
vol. 52, Amer. Math. Soc., Providence, RI, 1991, pp. 93每114

\bibitem[EGH10]{EFH10} Erdenberger, C.; Grushevsky, S.; Hulek, K. :{\it Some intersection numbers of divisors on toroidal compactifications of Ag.} J. Algebraic Geom. 19 (2010), no. 1, 99每132.

\bibitem[EV]{EV} Esnault, H. ; Viehweg, E. :{\it Lectures on vanishing theorems.} DMV Seminar, {\textbf{20}}.
Birkh\"auser Verlag, Basel, 1992. vi+164 pp.

\bibitem[FC]{FC}Faltings, G. \& Chai, C-L :{\it Degeneration of abelian varieties.
With an appendix by David Mumford.} Ergebnisse der Mathematik und ihrer Grenzgebiete (3) [Results in Mathematics and Related Areas (3)], 22. Springer-Verlag, Berlin, 1990.


\bibitem[Fre]{Fre} Freitag, E. :{\it Siegelsche Modulfunktionen.}
Grundlwhrem der Mathematischen Wissenschaften
254.Springer-Verlag,Berlin, 1983.


\bibitem[GS96]{GS96} Gritsenko, V. A. ; Sankaran, G. K. :{\it
Moduli of abelian surfaces with a (1,p2) polarisation. (English summary)} Izv. Ross. Akad. Nauk Ser. Mat. 60 (1996), no. 5, 19--26; translation in
Izv. Math. 60 (1996), no. 5, 893每900

\bibitem[GHS07]{GHS07} Gritsenko, V. A.; Hulek, K.; Sankaran, G. K. :{\it The Kodaira dimension of the moduli of K3 surfaces.} Invent. Math. 169 (2007), no. 3, 519每567.

\bibitem[GHS11]{GHS11} Gritsenko, V. A.; Hulek, K.; Sankaran, G. K. {\it  Moduli Spaces of Polarized Symplectic O'Grady Varieties and Brocherds Products.} J. Differential Geom. Volume 88, Number 1 (2011), 61-85.

\bibitem[Grit95]{Grit95}Gritsenko, V. :{\it
Irrationality of the moduli spaces of polarized abelian surfaces. (English summary)
With an appendix by the author and K. Hulek.} Abelian varieties (Egloffstein, 1993), 63每84, de Gruyter, Berlin, 1995.

\bibitem[Gri]{Gri} {\it Topics in transcendental algebraic geometry.} Proceedings of a seminar held at the Institute for Advanced Study, Princeton, N.J., during the academic year 1981/1982. Edited by Phillip Griffiths. Annals of Mathematics Studies, 106. Princeton University Press, Princeton, NJ, 1984.



%\bibitem[Hart]{Hart} Hartshorne,R. :  {\it  Algebraic geometry.}  Graduate Texts
%in Mathematics  {\bf 52}, Springer-Verlag, 1977.

\bibitem[Har]{Har} Harder,G. : {\it Eisensteinkohomologie und die Konstruktion gemischter Motive.}
Lecture Notes in Mathematics, 1562. Springer-Verlag, Berlin, 1993.

\bibitem[Hir63]{Hir63} Hironaka, H. :{\it Resolution of singularities of an algebraic variety over a field of characteristic zero. I, II.} Ann. of Math. (2) 79 (1964), 109每203; ibid. (2) 79 1964 205每326.

\bibitem[Hul00]{Hul00}  Hulek, K. :{\it Nef divisors on moduli spaces of abelian varieties.} Complex analysis and algebraic geometry, 255每274, de Gruyter, Berlin, 2000.

\bibitem[Igu64]{Igu64} Igusa, J. :{\it On Siegel modular forms of genus two.} Amer. J. Math. 84 1962 175--200.

\bibitem[Igu67]{Igu67} Igusa, J. :{\it A desingularization problem in the theory of Siegel modular functions.} Math. Ann. 168 1967 228--260.

\bibitem[Igu81]{Igu81} Igusa, J. :{\it Schottky's invariant and quadratic forms.} Christoffel Symposium, Birkh\"auser Verlag, 1981, 352-362.

\bibitem[Iitaka77]{Iitaka77} Iitaka, S. :{\it  On logarithmic Kodaira dimension of algebraic varieties.}
Complex analysis and algebraic geometry, pp. 175--189. Iwanami
Shoten, Tokyo, 1977.

\bibitem[Kawa81]{Kawa81} Kawamata,  Y. :{\it Characterization of abelian varieties.} Compositio Math.
43 (1981),  no. 2,  253--276.



\bibitem[Kol87]{Kol87} Koll\'ar, J. :{\it Subadditivity of the Kodaira
Dimension: Fibres of general type.} Algebraic Geometry, Sendai,
1985. Advanced Studies in Pure Mathematics {\bf 10} (1987)
361--398

\bibitem[Kon93]{Kon93}Kondo, S.: {\it On the Kodaira dimension of the moduli space of K3 surfaces.} Compos.
Math. 89, 251每299 (1993)


\bibitem[MM]{MM} Ma, X.-N. \&  Marinescu, G. {\it Holomorphic Morse Inequalities and Bergman  kernels.}Progress in Mathematics, 254. Birkh??user Verlag, Basel, 2007.

\bibitem[MVZ07]{MVZ07}Moeller,M.; Viehweg,E.; Zuo,K. {\it
Stability of Hodge bundles and a numerical characterization of Shimura varieties} J. Differential Geom. 92(1) (2012), 71–151.

\bibitem[Mil]{Mil} Milne, J.S. :{\it Alegraic Groups, Lie Groups, and their Arithmetic Subgroups.} Version 2.00, 2010. Unpublished book at website
http://www.jmilne.org/math/




\bibitem[Mum77]{Mum77} Mumford,D. :{\it Hirzebruch's proportionality theorem in
the non-compact case.} Inventiones math., {\bf 42} (1977), 239-272.
\bibitem[Mum83]{Mum83}Mumford,D. :{\it On the Kodaira dimension of the Siegel modular variety.}
 Algebraic geometry---open problems (Ravello, 1982), 348--375,
Lecture Notes in Math., 997, Springer, Berlin, 1983.


\bibitem[Poo96]{Poo96} Poor, C. :{\it Schottky's form and the hyperelliptic locus.} Proc. Amer. Math. Soc. 124 (1996), no. 7, 1987--1991.


\bibitem[Roy80]{Roy80}H.L. Royden  :{\it The Ahlfors-Schwarz lemma in several complex variables.} Comment.
Math. Helv. 55 (1980), no. 4, 547?58.
\bibitem[Sak77]{Sak77}Sakai, F. :{\it  Kodaira dimensions of complements of divisors.} Complex
analysis and algebraic geometry, pp. 239--257. Iwanami
Shoten,Tokyo,1977.
\bibitem[Sat]{Sat} Satake, I. :{\it Algebraic structures of symmetric domains.} $\rm{ Kan\hat o}$ Memorial Lectures, 4. Iwanami Shoten, Tokyo; Princeton University Press, Princeton, N.J., 1980.
\bibitem[Sch73]{Sch73} Schmid, W. :{\it Variation of Hodge Structure: The singularities
of the period mapping.} Invent. math. {\bf 22} (1973) 211--319
\bibitem[Sim88]{Sim88} Simpson, C. :{\it Constructing variations of Hodge Structure using Yang-Mills theory
and applications to uniformization.} Journal of the AMS {\bf
1}(1988) 867-918
\bibitem[Sim90]{Sim90} Simpson, C. :{\it Harmonic bundles on noncompact curves.
Journal of the AMS} {\bf 3} (1990) 713--770


\bibitem[SY82]{SY82} Siu, Y-T.\& Yau, S-T :{\it Compactification of negatively curved complete K\"ahler manifolds of finite volume.} Seminar on Differential Geometry, pp. 363--380, Ann. of Math. Stud., 102, Princeton Univ. Press, Princeton, N.J., 1982.
\bibitem[Siu93]{Siu93} Siu, Y.-T. :{\it An effective Matsusaka big theorem.} Ann.Inst.Fourier(Grenoble) 43(1993),no. 5, 1387-1405.

\bibitem[Tai82]{Tai82} Tai, Y. :{\it On the Kodaira dimension of the moduli space of abelian varieties.}  Invent. Math.  68  (1982), no. 3, 425--439.
\bibitem[TY86]{T-Y86} Tian, G.; Yau, S.-T. :{\it Existence of K\"ahler-Einstein metrics on complete K\"ahler manifolds and their applications to algebraic geometry.}  Mathematical aspects of string theory (San Diego, Calif., 1986),  574--628, Adv. Ser. Math. Phys., 1, World Sci. Publishing, Singapore, 1987.
\bibitem[Tsu86]{Tsu86} Tsuyumine, S. :{\it On Siegel modular forms of degree three.} Amer. J. Math. 108 (1986), no. 4, 755--862.


\bibitem[vdG08]{vdG08} van der Geer, G. :{\it Siegel modular forms and their applications.} (English summary) The 1-2-3 of modular forms, 181--245, Universitext, Springer, Berlin, 2008.

\bibitem[Wang93]{Wang93} Wang, W. :{\it On the smooth compactification of Siegel spaces.}
J. Differential Geom. 38 (1993), no. 2, 351--386.

\bibitem[Yau78]{Yau78} Yau, S.-T. :{\it On the Ricci curvature of
a compact K\"ahler manifold and the complex Monge-Amp\`ere
equation}, I, Comm. Pure Appl. Math.  31 (1978), 339-411.

\bibitem[Yau78-2]{Yau78-2} S.-T. Yau :{\it A general Schwarz lemma for K災hler manifolds.} Amer. J. Math.
100(1978), no. 1, 1970-203.
\bibitem[Yau87]{Yau87} Yau, S.-T. :{\it Uniformization of geometric structures.} The mathematical heritage
of Hermann Weyl (Durham, NC, 1987), 265--274, Proc. Sympos. Pure
Math., 48, Amer. Math. Soc., Providence, RI, 1988.
\bibitem[Yau93]{Yau93} Yau, S.-T. :{\it A splitting theorem and an algebraic
geometric characterization of locally Hermitian symmetric spaces.}
Comm. Anal. Geom. 1 (1993), no. 3-4, 473--486.


\bibitem[YZ11]{Y-Z} Yau, S.-T. ;  Zhang, Y. :{\it The Geometry on Smooth Toroidal Compactifications of Siegel varieties.} (2011), preprint arXiv:1201.3785.
\bibitem[Zuc81]{Zuc81} Zucker, S. :{\it Locally homogeneous variations of Hodge structure.} Enseign. Math. (2) 27 (1981), no. 3-4, 243-276 (1982).
\bibitem[Zuo00]{Zuo00} Zuo, K. :{\it On the negativity of kernels of
Kodaira-Spencer maps on Hodge bundles and applications.}  Asian
J. of Math.  {\bf 4} (2000) 279-302.

\end{thebibliography}
\end{document}